\documentclass{amsart}

\usepackage{amsmath, amssymb, amsfonts, amsthm, mathtools, mathrsfs}
\usepackage[all,cmtip,2cell]{xy}
\UseAllTwocells
\usepackage{tikz-cd}
\tikzcdset{scale cd/.style={every label/.append style={scale=#1},
    cells={nodes={scale=#1}}}}
\usepackage{graphicx}
\usepackage{hyperref}
\usepackage[T1]{fontenc}

\counterwithin{equation}{section}

\newcommand{\pbcorner}[1][ul]{\save*!/#1+1.6pc/#1:(1,-1)@^{|-}\restore}
\newcommand{\pocorner}[1][dr]{\save*!/#1+1.6pc/#1:(1,-1)@^{|-}\restore}

\newcommand{\simd}{\mathrel{\rotatebox[origin=c]{-90}{$\sim$}}}
\newcommand{\simu}{\mathrel{\rotatebox[origin=c]{90}{$\sim$}}}

\theoremstyle{plain}
\newtheorem{thm}{Theorem}[section]
\newtheorem{lem}[thm]{Lemma}
\newtheorem{prop}[thm]{Proposition}
\newtheorem{cor}[thm]{Corollary}
\newtheorem{conj}[thm]{Conjecture}
\newtheorem{que}[thm]{Question}

\newtheorem*{thm*}{Theorem}
\newtheorem*{lem*}{Lemma}
\newtheorem*{prop*}{Proposition}
\newtheorem*{cor*}{Corollary}
\newtheorem*{rem*}{Remark}
\newtheorem*{conj*}{Conjecture}

\theoremstyle{definition}
\newtheorem{dfn}[thm]{Definition}

\newtheorem{rem}[thm]{Remark}

\newtheorem*{dfn*}{Definition}
\newtheorem*{ex*}{Example}
\newtheorem*{NaC}{Notation and Convention}
\newtheorem*{ACK}{Acknowledgement}

\newcommand\ab{\allowbreak}
\newcommand\bs{\boldsymbol}

\newcommand{\bA}{{\mathbb A}}

\newcommand{\bC}{{\mathbb C}}
\newcommand{\bD}{{\mathbb D}}

\newcommand{\bL}{{\mathbb L}}

\newcommand{\bP}{{\mathbb P}}
\newcommand{\bQ}{{\mathbb Q}}
\newcommand{\bR}{{\mathbb R}}

\newcommand{\bZ}{{\mathbb Z}}

\newcommand{\cC}{{\mathcal C}}

\newcommand{\cF}{{\mathcal F}}
\newcommand{\cG}{{\mathcal G}}

\newcommand{\cL}{{\mathcal L}}

\newcommand{\cO}{{\mathcal O}}

\newcommand{\cR}{{\mathcal R}}

\newcommand{\cZ}{{\mathcal Z}}

\newcommand{\fC}{{\mathfrak C}}

\newcommand{\bfT}{\mathbf{T}}

\newcommand{\vphi}{\varphi}

\DeclareMathOperator{\Perf}{Perf}
\DeclareMathOperator{\Perv}{Perv}

\newcommand\Cone{\mathop{\mathrm{Cone}}\nolimits}

\newcommand{\sHom}{\mathop{\mathcal{H}\! \mathit{om}}\nolimits}
\newcommand{\sRHom}{\mathop{\mathit{R}\mathcal{H}\! \mathit{om}}\nolimits}
\DeclareMathOperator{\FS}{FS}

\newcommand\Spec{\mathop{\mathrm{Spec}}\nolimits}
\newcommand\dSpec{\mathop{\mathbf{Spec}}\nolimits}

\newcommand\Sym{\mathop{\mathrm{Sym}}\nolimits}

\newcommand\Tot{\mathop{\mathrm{Tot}}\nolimits}

\DeclareMathOperator{\mH}{\mathrm{H}}
\newcommand{\HBM}{\mH^{\mathrm{BM}}}

\newcommand\Crit{\mathop{\mathrm{Crit}}\nolimits}

\newcommand\id{\mathop{\mathrm{id}}\nolimits}

\newcommand\rank{\mathop{\mathrm{rank}}\nolimits}

\renewcommand\det{\mathop{\mathrm{det}}\nolimits}
\DeclareMathOperator{\vdim}{vdim}

\renewcommand{\Re}{\mathop{\mathrm{Re}}}

\newcommand{\red}{\mathrm{red}}
\newcommand\vir{\mathrm{vir}}
\DeclareMathOperator{\pt}{pt}

\DeclareMathOperator{\DT4}{DT4}
\DeclareMathOperator{\Sp}{Sp}
\DeclareMathOperator{\Bl}{Bl}
\DeclareMathOperator{\BM}{BM}
\DeclareMathOperator{\Chow}{Chow}
\DeclareMathOperator{\cl}{cl}

\newcommand{\pr}{\mathrm{pr}}
\newcommand{\bCu}{\widetilde{\bC^*}}

\newcommand{\FSS}{\mathop{\check{\mathrm{FS}}}\nolimits}
\newcommand{\FSn}{\mathop{\widetilde{\mathrm{FS}}}\nolimits}

\newcommand{\DR}{D_{\bR^+}}

\usepackage{subfiles}

\title{Virtual classes via vanishing cycles}
\author{Tasuki Kinjo}

\address{graduate school of mathematical science, the university of tokyo, 3-8-1 komaba,
meguroku, tokyo 153-8914, japan.}
\email{tasuki.kinjo@ipmu.jp}

\thanks{This work was supported by the WINGS-FMSP
program at the Graduate School of Mathematical Science, the University of Tokyo, and JSPS KAKENHI Grant number JP21J21118.}

\begin{document}

\begin{abstract}
We develop a new method to construct the virtual fundamental classes for quasi-smooth derived schemes using the perverse sheaves of vanishing cycles on their $-1$-shifted contangent spaces. It is based on the author's previous work that can be regarded as a version of the Thom isomorphism for $-1$-shifted cotangent spaces. We use the Fourier--Sato transform to prove that our classes coincide with the virtual fundamental classes introduced by the work of Behrend--Fantechi and Li--Tian, under the quasi-projectivity assumption.
We also discuss an approach to construct DT4 virtual classes for $-2$-shifted symplectic derived schemes using the perverse sheaves of vanishing cycles.

\end{abstract}

\maketitle

\setcounter{tocdepth}{1}
\tableofcontents


\section{Introduction}

\subsection{History}

In many cases, moduli spaces such as those of stable maps and coherent sheaves have bigger dimensions than the expected dimensions computed by the Riemann--Roch theorem.
To define a correct enumerative invariant counting points in such  a moduli space, we need to integrate cohomology classes over the \emph{virtual fundamental class} which is the cycle class of the expected dimension instead of the usual fundamental class which may be ill-defined.
In the work of Behrend--Fantechi \cite{BF97} and Li--Tian \cite{LT98}, they introduced the notion of perfect obstruction theory for a scheme (or more generally a Deligne-Mumford stack) $X$ and
constructed the virtual fundamental class
\[
[X]^{\vir} \in A_{\mathrm{exp.dim} X}(X)
\]
 from that data.
In the modern term, the notion of perfect obstruction theory is understood as a classical shadow of the quasi-smooth derived structure.
Their work can be used to construct Gromov--Witten invariants  counting stable maps to smooth projective varieties \cite{Beh97, LT98} and Donaldson--Thomas invariants  counting stable sheaves on Calabi--Yau threefolds \cite{Tho00}.

Later, Behrend \cite{Beh09} proved that for a projective variety $X$ defined over $\bC$ and a perfect obstruction theory on it admitting a certain symmetry (called symmetric perfect obstruction theory),
there exists a natural constructible function $\nu_X$ such that the equality
\begin{equation}\label{eq:intro1}
\int_{[X]^{\vir}} 1 = \sum_i i \cdot \chi(\nu_X^{-1}(i))
\end{equation}
holds. The natural perfect obstruction theory on a projective component of the moduli space of stable sheaves on a Calabi--Yau threefold satisfies this assumption.
This suggests that Donaldson--Thomas invariants admit sheaf theoretic categorifications.
This was achieved by the work of Joyce and his collaborators \cite{BBDJS15, BBJ19, BBBBJ15} which we briefly explain now.
Let $\bs{X}$ be a derived scheme equipped with a $-1$-shifted symplectic structure \cite{PTVV13}.
It induces a natural symmetric perfect obstruction theory on $X = t_0(\bs{X})$.
It is proved in \cite{BBJ19} that $\bs{X}$ is Zariski locally equivalent to a derived critical locus of a regular function on a smooth scheme.
If $\bs{X}$ admits an orientation, one can glue certain twists of the locally defined vanishing cycle complexes to define a globally defined perverse sheaf
\[
\varphi_{\bs{X}} \in \Perv(X).
\]
 See \cite{BBDJS15} for the detail. The Euler characteristic of the perverse sheaf $\varphi_{\bs{X}}$ at a point $x \in X$ is equal to the value of the Behrend function $\nu_X (x)$.
Therefore the equality \eqref{eq:intro1} implies the following equality
\[
\int_{[X]^{\vir}} 1 = \sum_i (-1)^i \mH^i(X, \varphi_{\bs{X}}).
\]
If we take $\bs{X}$ to be the derived moduli space of stable sheaves on a Calabi--Yau threefold equipped with the canonical $-1$-shifted symplectic structure \cite[Theorem 0.1]{PTVV13} and the canonical orientation constructed in \cite{JU21}, the perverse sheaf $\varphi_{\bs{X}}$ can be regarded as a categorification of the Donaldson--Thomas invariant.

This paper aims to connect two mathematical objects we have seen: the virtual fundamental class for an \emph{arbitrary} quasi-smooth derived scheme defined over $\bC$ (not necessary $-1$-shifted symplectic!) and the perverse sheaf associated with an oriented $-1$-shifted symplectic derived scheme.

\subsection{Results}

Let $\bs{Y}$ be a quasi-smooth derived scheme over $\Spec \bC$ and 
\[
\bfT^*[-1]\bs{Y} \coloneqq \dSpec_{\bs{Y}}(\Sym (\bL_{\bs{Y}}^\vee[1]))
\]
be its $-1$-shifted cotangent space.
Write $Y = t_0(\bs{Y})$ and $\widetilde{Y} = t_0(\bfT^*[-1]\bs{Y})$ and we let $\pi \colon \widetilde{Y} \to Y$ denote the natural projection.
It is shown in \cite[Proposition 1.21]{PTVV13} that $\bfT^*[-1]\bs{Y} \coloneqq \dSpec_{\bs{Y}}(\Sym (\bL_{\bs{Y}}^\vee[1]))$ carries a canonical $-1$-shifted symplectic structure and one can easily show that it also carries a natural orientation.
Therefore we can define a natural perverse sheaf $\varphi_{\bfT^*[-1]\bs{Y}} \in \Perv(\widetilde{Y})$.
The following isomorphisms are proved in \cite[Theorem 3.1]{Kin21}:
\begin{equation}\label{eq:dimredintro}
\gamma \colon \pi_! \varphi_{\bfT^*[-1]\bs{Y}} \cong \bQ_Y[\vdim \bs{Y}], \ \bar{\gamma} \colon \pi_* \varphi_{\bfT^*[-1]\bs{Y}} \cong \omega_Y[-\vdim \bs{Y}].
\end{equation}
Roughly speaking, the map $\gamma$ can be regarded as a version of the Thom isomorphism for $\bfT^*[-1]\bs{Y}$ and the map $\bar{\gamma}$ is an analog of the homotopy invariance for the singular cohomology.
Therefore the following composition
\[
\bQ_Y[\vdim \bs{Y}] \xrightarrow[\cong]{\gamma^{-1}} \pi_! \varphi_{\bfT^*[-1]\bs{Y}} \to \pi_* \varphi_{\bfT^*[-1]\bs{Y}} \xrightarrow[\cong]{\bar{\gamma}} \omega_Y[-\vdim \bs{Y}]
\]
can be regarded as the Euler class for $\bfT^*[-1]\bs{Y}$. We let
\[
e_{\varphi}(\bfT^*[-1]\bs{Y}) \in \HBM_{2\vdim \bs{Y}}(Y)
\]
be the class corresponding to the above composition.
As the virtual fundamental class is a generalization of the Euler class, it is natural to compare the class $e_{\varphi}(\bfT^*[-1]\bs{Y})$
with the virtual fundamental class $[\bs{Y}]^{\vir}$. The aim of this paper is to prove the following result:

\begin{thm}[= Theorem \ref{thm:main}]\label{thm:intro}
  Assume $\bL_{\bs{Y}} |_Y$ has a global resolution by a two-term complex of locally free sheaves. Then the following equality in $\HBM_{2\vdim{\bs{Y}}}(Y)$ holds:
  \[
  e_{\varphi}(\bfT^*[-1]\bs{Y}) = (-1)^{\vdim \bs{Y} \cdot (\vdim \bs{Y} -1)/2}\cl_{Y}([\bs{Y}]^{\vir})
  \]
  where $\cl_Y \colon A_{\vdim \bs{Y}}(Y) \to \HBM_{2\vdim{\bs{Y}}}(Y)$ is the cycle class map.
\end{thm}

In other words, the above theorem gives a new construction of the virtual fundamental class in the Borel--Moore homology under a mild assumption that is satisfied when $Y$ is quasi-projective.

\subsection{Strategy of the proof}\label{ssec:strategy}

Before explaining the strategy of the proof of Theorem \ref{thm:intro}, let us recall the construction of the virtual fundamental class due to \cite{BF97, LT98}.
Firstly, the natural map $\bL_{\bs{Y}} |_Y \to \bL_Y$ defines a purely zero dimensional cone substack $\fC_{Y} \subset t_0(\bfT[1]\bs{Y})$ of the classical truncation of the $1$-shifted tangent stack called the intrinsic normal cone of $Y$.
Then the virtual fundamental class is defined to be the image of the class $[\fC_Y] \in A_0(t_0(\bfT[1]\bs{Y}))$ under the natural isomorphism $A_0(t_0(\bfT[1]\bs{Y})) \cong A_{\vdim \bs{Y}}(Y)$.

The main difficulty of Theorem \ref{thm:main} is that proving the statement locally does not imply the global statement since the Borel--Moore homology does not satisfy the sheaf condition (in the $1$-categorical sense).
Therefore we need to compare the intrinsic normal cone $\fC_{Y}$ and the vanishing cycle complex $\varphi_{\bfT^*[-1]\bs{Y}}$ directly though they live in different spaces.

We overcome this difficulty using the Fourier--Sato transform\footnote{The author came up with this idea after a private discussion with Adeel Khan where he learned Khan's ongoing project on derived microlocal geometry. We believe that the technical assumption on the global resolution in Theorem \ref{thm:intro} can be removed once basic properties of the Fourier--Sato transform is proved in this context.} to relate constructible sheaves on $\bfT^*[-1]\bs{Y}$ and those on $\bfT[1]\bs{Y}$.
More precisely take a global resolution $\bL_{\bs{Y}}|_Y \cong [E^{-1} \to E^0]$ by locally free sheaves and consider the Fourier--Sato transform \cite[\S 3.7]{KS13}
\[
\FS_{E^{-1}} \colon D_{\bR^+}^b (E^{-1}) \to  D_{\bR^+}^b ((E^{-1})^\vee).
\]
Using the natural inclusion $t_0(\bfT^*[-1]\bs{Y}) \hookrightarrow E^{-1}$, one can regard $\varphi_{\bfT^*[-1]\bs{Y}}$ as a perverse sheaf on $E^{-1}$.
Then one can show that the complex $\FS_{E^{-1}}(\varphi_{\bfT^*[-1]\bs{Y}})$ has the support contained in the pullback of the intrinsic normal cone $p^* \fC_Y$ along the natural map $p \colon (E^{-1})^\vee \to [(E^{-1})^\vee / (E^0)^\vee ] \cong t_0(\bfT[1]\bs{Y}) $ (see Lemma \ref{lem:keying}).
This observation and some sheaf theoretic analysis imply that Theorem \ref{thm:intro} holds.

\subsection{Microlocal interpretation}

The proof of Theorem \ref{thm:main} outlined in the previous section is motivated by microlocal geometry. We now briefly explain it.

Let $X$ be a manifold and $M \subset X$ be a closed submanifold.
Kashiwara--Schapira \cite[\S 4.3]{KS13} introduced a functor
\[
\mu_{M/X} \colon D^b(X) \to D^b_{\bR^+}(T_M^* X)
\]
called the microlocalization functor. It is defined as the Fourier--Sato dual of the specialization functor
\[
\Sp_{M/X} \colon D^b(X) \to D^b_{\bR^+}(T_M X).
\]

For a quasi-smooth derived scheme $\bs{Y}$, the $-1$-shifted cotangent scheme $\bfT^*[-1]\bs{Y}$ can be regarded as the ``conormal bundle'' along the constant map $\bs{Y} \to *$. Then the following slogan is the main idea behind the proof of  Theorem \ref{thm:main}:

\begin{quote}
    Slogan: The vanishing cycle complex $\varphi_{\bfT^*[-1]\bs{Y}}$ can be thought of as the microlocalization of the constant sheaf on a point along the constant map $\bs{Y} \to *$.
\end{quote}
In other words, the vanishing cycle complex $\varphi_{\bfT^*[-1]\bs{Y}}$ should be thought of as the  ``intrinsic microlocalization''.
Here the usage of the term ``intrinsic'' is completely analogous to the one in the intrinsic normal cone.
We can relate them using the Fourier--Sato transform.

Unfortunately, the above slogan is not mathematically rigorous (at least now) because we do not have  microlocalization functors with respect to arbitrary morphisms between derived schemes. Nevertheless, using a local model for quasi-smooth derived schemes, we can formulate a local version of the above slogan which will be stated in Theorem \ref{thm:vansp}.
This local version is enough for our purposes.

Recently, Kendric Schefers \cite{Sch22} gave a completely different proof of Theorem \ref{thm:vansp}.
Further, he introduced a microlocalization functor for quasi-smooth closed immersion and realized the above slogan when $\bs{Y}$ can be embedded into a smooth ambient space.
In another work, Adeel Khan recently announced a construction of the microlocalization functor for arbitrary morphisms between derived Artin stacks with some technical assumptions.
We believe that the above slogan can be made precise using his mircolocalization functor.

\subsection{Joyce's conjecture}

The present work is motivated by Joyce's conjecture \cite[Conjecture 5.18]{AB17} which we briefly discuss now.
Let $\bs{\tau} \colon \bs{L} \to \bs{X}$ be an oriented Lagrangian of an oriented $-1$-shifted symplectic derived Artin stack.
Write $\tau = t_0(\bs{\tau})$, $L = t_0(\bs{L})$ and $X = t_0(\bs{X})$.
Let $\varphi_{\bs{X}} \in \Perv(X)$ be the vanishing cycle complex associated with the oriented $-1$-shifted symplectic structure on $\bs{X}$.
Then Joyce's conjecture predicts that there exists a natural map
\[
\mu_{\bs{L}} \colon \bQ_L[\vdim \bs{L}] \to \tau^! \varphi_{\bs{X}}.
\]
See \cite[\S 5.3]{AB17} for the full statement of this conjecture and \cite[\S 7]{Joyslide} for important consequences including the construction of cohomological Hall algebras for CY3 categories.
When $\bs{X}$ is a point equipped with the trivial $-1$-shifted symplectic structure, the vanishing cycle complex $\varphi_{\bs{X}}$ is the constant sheaf and the oriented Lagrangian structure for $\bs{\tau}$ is identified with an oriented $-2$-shifted symplectic structure on $\bs{L}$.
In this case, it is expected that the map $\mu_{\bs{L}}$ is given by the DT4 virtual class introduced in the work of Cao--Leung \cite{CL14}, Borisov--Joyce \cite{BJ17} and Oh--Thomas \cite{OT20}.
In other words, the map $\mu_{\bs{L}}$ in the general case should be understood as a relative version of the DT4 virtual class.

Joyce's conjecture suggests that (perverse) sheaf theory may be used to give a new construction of DT4 virtual classes.
Since the virtual fundamental class for a quasi-smooth derived scheme is equal to the DT4 virtual class of its $-2$-shifted cotangent space (see \cite[\S 8]{OT20}), Theorem \ref{thm:intro} can be regarded as the first step toward this project.
In Section \ref{ssec:DT4}, we propose a conjectural approach to construct DT4 virtual classes  assuming some variations of the isomorphism \eqref{eq:dimredintro}.

\subsection{Structure of the paper}

The paper is organized as follows.

In Section \ref{sec:app} we recall some sheaf operations and prove some of their basic properties.

In Section \ref{sec:sp} we recall the construction of Verdier's specialization functor and use it to construct the specialization morphism of the Borel--Moore homology.

In Section \ref{sec:FS} we prove that the Fourier--Sato transform relates the specialization functor and the vanishing cycle functor.
 This gives a new proof of Davison's dimensional reduction theorem \cite[Theorem A.1]{Dav17} and a proof of Theorem \ref{thm:intro} for local models of quasi-smooth derived schemes.
 
 In Section \ref{sec:main} we prove Theorem \ref{thm:intro}.
 
 In Section \ref{sec:future} we discuss some conjectural generalizations of Theorem \ref{thm:intro}.

\begin{ACK}
  I am very grateful to Adeel Khan for sharing many ideas on his ongoing project on derived microlocal geometry.
  Without him, I could have never come up with the idea to use the Fourier--Sato transform to prove Theorem \ref{thm:intro}.
  I also thank my supervisor Yukinobu Toda for fruitful discussions and careful reading,  Hyeonjun Park for teaching me DT4 theory and Alexandre Minets for pointing out that a conjecture in the previous version of this paper was not correct.
\end{ACK}

\begin{NaC}

\begin{itemize}

\item All schemes considered in this paper are separated and of finite type over the complex number field.

\item For a topological space $X$, we let $D^+(X)$ (resp. $D^b(X)$, $D^b _c(X)$) denote the derived category of lower bounded complexes (resp. bounded complexes, bounded complexes with constructible cohomology) in $\bQ$-vector spaces.

\item For a closed subset $Z \subset X$ and a complex $\cF \in D^+(Z)$, we let
      $\cF |_Z$ denote the complex $(Z \hookrightarrow X)^* \cF$,
      $\cF_Z$ denote the complex $(Z \hookrightarrow X)_*(Z \hookrightarrow X)^* \cF$,
      $\cF |_Z^!$ denote the complex $(Z \hookrightarrow X)^! \cF$ and
      $\Gamma_Z(\cF)$ denote the complex $(Z \hookrightarrow X)_!(Z \hookrightarrow X)^! \cF$.

\item For a closed subset $Z \subset X$ and a complex $\cF \in D^+(Z)$, we sometimes regard $\cF$ as a complex on $X$.
Conversely, for a complex $\cG \in D^+(X)$ whose support is contained in $Z$, we sometimes regard $\cG$ as a complex on $Z$.

\item For a continuous map $f \colon X \to Y$ between locally compact Hausdorff spaces such that the functor $f^!$ is defined, we let $\omega_{X/Y}$ denote the complex $f^! \bQ_{Y}$ and
$\mH^*(X \to Y)$ denote its cohomology $\mH^*(X, \omega_{X/Y})$. We also write $\omega_X = \omega_{X / \mathrm{pt}}$ and $\HBM_*(X) = \mH^{-*}(X \to \mathrm{pt})$. The Verdier-duality functor $\sRHom(-, \omega_X)$ is denoted by $\bD_X$.

\item For a complex separated scheme $X$ of finite type and its irreducible subscheme $Z$, the classes $[Z] \in A_{\dim Z}(X)$ and $[Z]_{\BM} \in \HBM_{2 \dim Z}(X)$ are defined by
the pushforward of the fundamental class of $Z$ to $X$ in the Chow group and the Borel--Moore homology group respectively.

\item
If there is no confusion, we use expressions such as $f_*$, $f_!$, and $\sHom$ for the derived functors
$f_*$, $f_!$, and $\sRHom$.

\end{itemize}

\end{NaC}


\section{Preliminaries on sheaf operations}\label{sec:app}

In this section, we recall the definition of the nearby cycle and vanishing cycle functors and the Fourier--Sato transform, and prove some of their basic properties. 
Readers who are familiar with sheaf theory can safely skip this section.
We assume that all topological spaces are locally compact Hausdorff, and for all continuous maps $f$, the functor $f_!$ has finite cohomological dimension.

\subsection{Base change maps}

Consider a commutative diagram of topological spaces
\begin{equation}\label{eq:square}
\begin{tikzcd}
& X' \arrow[r,"f'"]\arrow[d,"g'"] & Y' \arrow[d,"g"]\\
& X \arrow[r,"f"] & Y.
\end{tikzcd}
\end{equation}
We have the natural transformations of functors from $D^+(X)$ to  $D^+(Y')$
\begin{equation}\label{eq:bc1}
  g^* f_* \to f'_* {g'}^*,
\end{equation}
and functors from $D^+(X)$ to  $D^+(Y')$
\begin{equation}\label{eq:bc2}
  f'_! {g'}^! \to g^! f_!.
\end{equation}
The former map is defined by the composition
\[
g^* f_* \to f'_* {f'}^* g^* f_* \cong f'_* {g'}^* f^* f_*  \to f'_* {g'}^*
\]
and the latter map is defined similarly.

If the diagram \eqref{eq:square} is Cartesian, we have following natural transformations of functors from $D^+(X')$ to  $D^+(Y)$
\begin{equation}\label{eq:bc3}
f_! g'_* \to  g_* f'_!,
\end{equation}
functors from $D^+(X)$ to  $D^+(Y')$
\begin{equation}\label{eq:bc4}
g^* f_! \xrightarrow{\sim} f'_! {g'}^*,
\end{equation}
functors from $D^+(X)$ to  $D^+(Y')$
\begin{equation}\label{eq:bc5}
  f'_* {g'}^! \xrightarrow{\sim} g^! f_*,
\end{equation}
and functors from  $D^+(Y)$ to  $D^+(X')$
\begin{equation}\label{eq:bc6}
{g'}^* f^! \to {f'}^! g^*.
\end{equation}
The map \eqref{eq:bc3} is defined in \cite[(2.5.7)]{KS13}, \eqref{eq:bc4} is defined in \cite[Proposition 2.6.7]{KS13},
\eqref{eq:bc5} is defined in \cite[Proposition 3.1.9 (ii)]{KS13} and \eqref{eq:bc6} is defined in \cite[Proposition 3.1.9 (iii)]{KS13}.
We call the maps \eqref{eq:bc1}–\eqref{eq:bc6} base change maps.

\subsection{Natural transformations between four functors}\label{sect:nattrans}

Let $f \colon X \to Y$ be a continuous map between topological spaces.
Then we can construct the following natural transformations
\begin{align}
  \id \to f_* f^*, &\ f^* f_* \to \id \label{eq:*unit} \\
  \id \to f^! f_! ,&\ f_! f^! \to \id \label{eq:!unit}\\
  f_! &\to f_* \label{eq:!*1}.
\end{align}
If $f$ is a closed immersion, we have a natural transformation
\begin{equation}\label{eq:!*2}
f^! \to f^*
\end{equation}
obtained by the composition
\[
f^! \cong f^* f_* f^! \cong f^* f_! f^! \to f^*.
\]
If we have an isomorphism $f^! \bQ_Y \cong \bQ_X[r]$ (e.g. a smooth morphism between complex analytic spaces),
we have a natural transformation
\begin{equation}\label{eq:smpurity}
f^*[r] \to f^!.
\end{equation}
This is a consequence of \cite[Proposition 3.1.11]{KS13}, which constructs a natural transformation
\begin{equation}\label{eq:purity}
  f^*(-) \otimes f^!\bQ_{Y} \to f^!(-)
\end{equation}
for general $f$.

Now consider a Cartesian diagram of topological spaces
\[
\xymatrix{
& X' \ar[r]^{f'} \pbcorner \ar[d]^{g'} & Y' \ar[d]^{g}\\
& X \ar[r]^{f} & Y.
}
\]
For a choice of the natural transformations $\eta$ in \eqref{eq:*unit}–\eqref{eq:!*2} and a choice of the sheaf operation $a \in \{(-)^*, (-)_*, (-)^!, (-)_!\}$,
we can verify the commutativity of $\eta$ and the base change maps for $a$. For example, if we take $\eta$ to be the unit map
$\id \to f_* f^*$ and $a = (-)^!$, the following diagram commutes:
\[
\xymatrix{
g^! \ar[r] \ar[d] & g^! {f}_* f^* \\
f'_* {f'}^* g^! \ar[r]^{} & f'_* {g'} ^! f^*. \ar[u]^-{\simd}
}
\]
Now assume that there exists an isomorphism $f^! \bQ_Y  \cong \bQ_X[r]$ for some integer $r$ and the following composition
\begin{align}\label{eq:upco}
\bQ_{X'}[r] \cong {g'}^*{\bQ_{X}}[r] \cong {g'}^* f^! {\bQ_{Y}} \to {f'}^! g^* {\bQ_{Y}} \cong {f'}^! {\bQ_{Y'}}
\end{align}
is isomorphic. Then we can see that for each choice of $a \in \{(-)^*, (-)_*, (-)^!, (-)_!\}$ the natural transformation \eqref{eq:smpurity} commutes with the base change maps for $a$.
For example if we take $a = (-)^!$, the following diagram commutes:
\[
\xymatrix{
{f'}^*[r] g^! \ar[r] \ar[d] & {f'}^! g^! \ar[d]^{\simd} \\
{g'}^! f^*[r] \ar[r] & {g'}^! f^!.
}
\]
Now remove the assumption that \eqref{eq:upco} is isomorphic and take
$a \in \{(-)_*, (-)_!\} $ and $b \in \{(-)^*, (-)^!\}$.
Let $G_a \colon D^+(Y') \to D^+(Y)$ (resp. $G_b \colon D^+(Y) \to D^+(Y')$) be the functor corresponding to $a$ (resp. $b$).
Then we can see that the base change maps for the composition transformation $G_aG_b$ commute with the map \eqref{eq:smpurity}.
For example if we take $a = (-)_*$ and $b = (-)^*$, or $a = (-)_!$ and $b = (-)^!$, we have the following commutative diagrams
\begin{align}\label{eq:puritybcc}
  \begin{split}
\xymatrix{
f^*[r] g_* g^*  \ar[r] \ar[d]& f^! g_* g^*\\
g'_* {f'}^*[r] g^*  \ar[d]_-{\simd}  & g'_* {f'}^! g^* \ar[u]^-{\simd} \\
g'_* {g'}^* f^*[r] \ar[r] & g'_* {g'}^* f^!\ar[u]
} \hspace{20pt}
\xymatrix{
f^*[r] g_! g^!  \ar[r] \ar[d]_-{\simd}& f^! g_! g^!\\
g'_! {f'}^*[r] g^! \ar[d]_-{}  & g'_! {f'}^! g^! \ar[u]^-{} \\
g'_! {g'}^! f^*[r] \ar[r] & g'_! {g'}^! f^!. \ar[u]^-{\simd}
}
\end{split}
\end{align}

\subsection{Vanishing cycle and nearby cycle functors}

Let $X$ be a complex analytic space and $u \colon X \to \bC$ be a complex analytic function.
We write
\[
X_0 \coloneqq u^{-1}(0), \ X_{>0} \coloneqq \{z \in X \mid \Re(u(z)) >0 \}, \ X_{\leq 0} \coloneqq X \setminus X_{>0}.
\]
We define the vanishing cycle functor and the nearby cycle functor
\[
\psi_{u}, \varphi_u \colon D^+(X) \to D^+(X_0)
\]
 by the formulas
\begin{align*}
\psi_u &\coloneqq (X_0 \hookrightarrow X)^* (X_{>0} \hookrightarrow X)_* (X_{>0} \hookrightarrow X)^*, \\
 \varphi_u &\coloneqq (X_0 \hookrightarrow X_{\leq 0})^* (X_{\leq 0} \hookrightarrow X)^!.
\end{align*}
If we are given a morphism $q \colon T \to X$ from another complex analytic space, we have natural transformations
\begin{align}
  \psi_u q_* \to {q_0}_* \psi_{u \circ q}&,\ \ \  \varphi_u q_* \to {q_0}_* \varphi_{u \circ q}, \label{eq:nvbc1} \\
  {q_0}_! \psi_{u \circ q} \to \psi_u q_! &,\ \ \  {q_0}_! \varphi_{u \circ q} \to \varphi_u q_!, \label{eq:nvbc2}\\
  q_0^* \psi_{u} \to \psi_{u \circ q} q^* &,\ \ \ q_0^* \varphi_{u} \to \varphi_{u \circ q} q^*, \label{eq:nvbc3}\\
 \psi_{u \circ q} q^! \to  q_0^! \psi_{u} &,\ \ \ \varphi_{u \circ q} q^! \to  q_0^! \varphi_{u} \label{eq:nvbc4},
\end{align}
where the map $q_0 \colon  T_0 = (u \circ q)^{-1}(0) \to X_0$ is induced from $q$.
All of these maps are defined by combining the base change maps \eqref{eq:bc1}–\eqref{eq:bc6}.
We call these maps the base change maps for the nearby and vanishing cycle functor.
The four maps in \eqref{eq:nvbc1} and \eqref{eq:nvbc2} are isomorphic when $q$ is proper,
and the four maps in \eqref{eq:nvbc3} and \eqref{eq:nvbc4} are isomorphic when $q$ is smooth.

There is another way to define the nearby cycle and vanishing cycle functors as in \cite[\S 8.6]{KS13}.
Let $\bCu$ the universal cover of $\bC^* = \bC \setminus \{0 \}$, and $p \colon \bCu \to \bC$ the natural map.
Define objects $K_{\psi}', K_{\phi}' \in D^b(\bC)$ by
\begin{align*}
K_{\psi}' \coloneqq & p_! \bQ_{\bCu} \\
K_{\varphi}'  \coloneqq & \Cone(p_! \bQ_{\bCu} \xrightarrow{\mathrm{tr}} \bQ_{\bC})
\end{align*}
where the map $\mathrm{tr} \colon p_! \bQ_{\bCu} \to \bQ_{\bC}$ is defined by the composition
\[
p_! \bQ_{\bCu} \cong p_! p^! \bQ_{\bC} \to \bQ_{\bC}.
\]
For a given complex analytic function $u \colon X \to \bC$, we define functors
\[
\psi_{u}', \varphi_u' \colon D^+(X) \to D^+(X_0)
\]
 by
\begin{align*}
\psi_u' &\coloneqq (X_0 \hookrightarrow X)^* \sRHom(u^* K_{\psi}', -), \\
\varphi_u' &\coloneqq (X_0 \hookrightarrow X)^* \sRHom(u^* K_{\varphi}', -).
\end{align*}
Now we want to compare $\psi_u$ and $\psi_u'$, and $\varphi_u$ and $\varphi_u'$.
To do this, we first define objects $K_{\psi}, K_{\phi} \in D^b(\bC)$ by
\begin{align*}
K_{\psi} \coloneqq & (\bC_{>0} \hookrightarrow \bC)_! \bQ_{\bC_{>0}} \\
K_{\varphi}  \coloneqq & \Cone((\bC_{>0} \hookrightarrow \bC)_! \bQ_{\bC_{>0}} \to \bQ_{\bC}) \cong
(\bC_{\leq 0} \hookrightarrow \bC)_! \bQ_{\bC_{\leq 0}}.
\end{align*}
Then we have isomorphisms
\begin{align*}
(X_0 \hookrightarrow X)^* \sRHom(u^* K_{\psi}, -)
& \cong (X_0 \hookrightarrow X)^* \sRHom((X_{>0} \hookrightarrow X)_! \bQ_{X_{>0}}, -)  \\
& \cong (X_0 \hookrightarrow X)^* (X_{>0} \hookrightarrow X)_* \sRHom(\bQ_{X_{>0}}, (X_{>0} \hookrightarrow X)^*(-)) \\
& \cong (X_0 \hookrightarrow X)^* (X_{>0} \hookrightarrow X)_* (X_{>0} \hookrightarrow X)^* = \psi_u
\end{align*}
and
\begin{align*}
  (X_0 \hookrightarrow X)^* \sRHom(u^* K_{\phi}, -)
  & \cong (X_0 \hookrightarrow X)^* \sRHom((X_{\leq 0} \hookrightarrow X)_! \bQ_{X_{\leq 0}}, -)  \\
  & \cong (X_0 \hookrightarrow X)^* (X_{\leq 0} \hookrightarrow X)_* \sRHom(\bQ_{X_{\leq 0}}, (X_{\leq 0} \hookrightarrow X)^!(-)) \\
  & \cong (X_0 \hookrightarrow X_{\leq 0})^* (X_{\leq 0} \hookrightarrow X)^! = \varphi_u.
\end{align*}
If we pick an open immersion $j \colon \bC_{>0} \hookrightarrow \bCu$ over $\bC$,
the natural map $j_! \bQ_{\bC>0} \to \bQ_{\bCu}$ induces maps
\[
K_{\psi} \to K'_{\psi},\ \ \ K_{\varphi} \to K'_{\varphi}
\]
and hence maps
\[
\psi_u' \to \psi_u,\ \ \ \varphi_u' \to \varphi_u.
\]
For a constructible complex $\cF \in D_c^{b}(X)$, one can show that the maps
\begin{equation}\label{eq:nvisom}
\psi_u'(\cF) \to \psi_u(\cF),\ \ \ \varphi_u'(\cF) \to \varphi_u(\cF)
\end{equation}
are isomorphic (see \cite[Lemma 1.1.1, Example 1.1.3]{Sch03} for the detail).

Now consider the following natural short exact sequence of complexes on $\bC$:
\begin{equation}\label{eq:shex1}
\Delta_1 \colon 0 \to \bQ_{\bC} \to K_{\varphi}' \to K_{\psi}'[1] \to 0.
\end{equation}
Then the distinguished triangle $(X_0 \hookrightarrow X)^* \sRHom(u^* \Delta_1, -)$ is identified with
\begin{equation}\label{eq:nvdist1'}
\psi_u'[-1] \to \varphi_u' \to (X_0 \hookrightarrow X)^* \to \psi_u'.
\end{equation}
Next, consider the following short exact sequence of complexes on $\bC$:
\begin{equation}\label{eq:shex2}
\Delta_2 \colon 0 \to K_{\psi}'[1] \to  K_{\varphi}' \to \Cone((\bC^* \hookrightarrow \bC)_! \bQ_{\bC^*} \to \bQ_{\bC}) \to 0.
\end{equation}
Here the first map is defined by the following diagram
\[
\xymatrix{
p_! \bQ_{\bCu} \ar[r]^{1-T} \ar[d] & p_! \bQ_{\bCu} \ar[d] \\
0 \ar[r] & \bQ_{\bC}
}
\]
where $T$ is the monodromy operator: In other words, if we let $\alpha \colon \bCu \xrightarrow{\sim} \bCu$ denote the covering transformation over $\bC^*$ corresponding to a counterclockwise loop, the map $T$ is given by the map
 $\alpha_! \bQ_{\bCu} \cong \bQ_{\bCu}$.
The second map in $\Delta_2$ is defined by the following diagram:
\[
\xymatrix{
 p_! \bQ_{\bCu} \ar[d] \ar[r]^-{\mathrm{tr}} & (\bC^* \hookrightarrow \bC)_! \bQ_{\bC^*}  \ar[d] \\
 \bQ_{\bC} \ar[r]^{\id} & \bQ_{\bC}.
 }
\]
Since we have an isomorphism $\Cone((\bC^* \hookrightarrow \bC)_! \bQ_{\bC^*} \to \bQ_{\bC}) \cong \bQ_{0}$,
the distinguished triangle $(X_0 \hookrightarrow X)^* \sRHom(u^* \Delta_2, -)$ is identified with
\begin{equation}\label{eq:nvdist2'}
(X_0 \hookrightarrow X)^! \to \varphi_u' \to \psi_u'[-1] \to (X_0 \hookrightarrow X)^! [1].
\end{equation}
For a constructible complex $\cF \in D^b_c(X)$, the isomorphisms in \eqref{eq:nvisom} and the distinguished triangles
\eqref{eq:nvdist1'} and \eqref{eq:nvdist2'}  induces the following distinguished triangles
\begin{align}
\psi_u(\cF)[-1] \to  \varphi_u(\cF) &\to (X_0 \hookrightarrow X)^*\cF \to \psi_u(\cF) \label{eq:nvdist1} \\
(X_0 \hookrightarrow X)^!\cF \to \varphi_u(\cF) &\to \psi_u(\cF)[-1] \to (X_0 \hookrightarrow X)^! \cF [1] \label{eq:nvdist2}
\end{align}
It is clear that the map
\begin{equation}\label{eq:v*}
\varphi_u(\cF) \to (X_0 \hookrightarrow X)^*\cF
\end{equation}
 in the distinguished triangle \eqref{eq:nvdist1} is equal to the following composition
\[
\varphi_u(\cF) = (X_0 \hookrightarrow X_{\leq 0})^* (X_{\leq 0} \hookrightarrow X)^! \cF \to (X_0 \hookrightarrow X_{\leq 0})^* (X_{\leq 0} \hookrightarrow X)^* \cF \cong (X_0 \hookrightarrow X)^*\cF
\]
where the first map is \eqref{eq:!*2}, and the map
\begin{equation}\label{eq:!v}
 (X_0 \hookrightarrow X)^! \cF \to \varphi_u(\cF)
\end{equation}
 in the distinguished triangle \eqref{eq:nvdist2}
is equal to the following composition
\[
(X_0 \hookrightarrow X)^! \cF \cong (X_0 \hookrightarrow X_{\leq 0})^! (X_{\leq 0} \hookrightarrow X)^! \to (X_0 \hookrightarrow X_{\leq 0})^* (X_{\leq 0} \hookrightarrow X)^! = \varphi_u(\cF)
\]
where the second map is \eqref{eq:!*2}.

Let $q \colon T \to X$ be a morphism of complex analytic spaces and $u$ be a holomorphic function on $X$.
Then for each choice of $a \in  \{(-)^*, (-)_*, (-)^!, (-)_!\}$, one can show that the distinguished triangles \eqref{eq:nvdist1} and \eqref{eq:nvdist2} commute with the base change maps for $a$.
For example, if we take $a = (-)_*$, for a complex $\cF \in D^b_c(T)$ the following diagrams commute:
\begin{align}
\begin{split}
  \xymatrix@C=8pt{
  \psi_u(q_*\cF)[-1] \ar[r] \ar[d] 
  & \varphi_u(q_*\cF) \ar[r] \ar[d] 
  & (X_0 \hookrightarrow X)^*q_*\cF \ar[r] \ar[d]
  & \psi_u(q_*\cF) \ar[d] \\
  {q_0}_* \psi_{u \circ q}(\cF)[-1] \ar[r] 
  & {q_0}_* \varphi_{u \circ q}(\cF) \ar[r]  
  & {q_0}_* (T_0 \hookrightarrow T)^*\cF \ar[r]  
  & {q_0}_*\psi_{u \circ q}(\cF),
  }
  \end{split}
\end{align}

\begin{align}
\begin{split}
  \xymatrix@C=7pt{
  (X_0 \hookrightarrow X)^!q_*\cF \ar[r] \ar[d]^-{\simd} 
  & \varphi_u(q_*\cF)  \ar[r] \ar[d] 
  & \psi_u(q_*\cF)[-1] \ar[r] \ar[d] 
  &  \ar[d]^-{\simd} (X_0 \hookrightarrow X)^!q_*\cF[1] \\
  {q_0}_* (T_0 \hookrightarrow T)^!\cF \ar[r] 
  & {q_0}_* \varphi_{u \circ q}(\cF) \ar[r] 
  & {q_0}_* \psi_{u \circ q}(\cF)[-1] \ar[r] 
  & {q_0}_* (T_0 \hookrightarrow T)^!\cF[1],
  }
\end{split}
\end{align}

Now let $Z$ be a complex analytic space and take a constructible object $\cF \in D^b_c(Z)$.
Denote by $\pi \colon Z \times \bC \to Z$ and $u \colon Z \times \bC \to \bC$ the projections.
Then it is clear that $\varphi_u(\pi^* \cF) = 0$.
We have the following statement whose proof will be given in \S \ref{ssec:nearbyid}:

\begin{prop}\label{prop:nearbyid}
Denote by $i \colon Z \times \{ 0 \} \hookrightarrow Z \times \bC$ the natural inclusion.
Then the following composition
\[
\cF \cong i^* \pi^* \cF \xrightarrow{\sim} \psi_u(\pi^* \cF) \xrightarrow{\sim} i^! \pi^* \cF [2] \cong i^! \pi^! \cF \cong \cF
\]
is the identity map.
\end{prop}

Now we discuss the commutativity of the vanishing cycle functor and the Verdier duality functor following \cite{Mas16}.
Firstly, for $\alpha, \beta \in \bR$ such that $0 \leq \beta - \alpha < 2 \pi$, we define a closed subset
$D^{[\alpha, \beta]} \subset \bC$ by
\[
D^{[\alpha, \beta]} \coloneqq \{ z = R \cdot e^{i \theta} \in \bC \mid R \in \bR_{\geq 0}, \theta \in [\alpha, \beta]  \}.
\]
For a complex analytic space $X$ and a regular function $u$ on it, we define a functor $\varphi_u^{[\alpha, \beta]} \colon D^+(X) \to D^+(X_0)$ by
\[
\varphi_{u}^{[\alpha, \beta]} \coloneqq (X_0 \hookrightarrow u^{-1}(D^{[\alpha, \beta]}))^* (u^{-1}(D^{[\alpha, \beta]}) \hookrightarrow X)^!.
\]
Clearly, we have $\varphi_u = \varphi_u^{[\pi/2, 3\pi/2]}$.
For real numbers $\alpha, \alpha', \beta' , \beta \in \bR$ such that $\alpha \leq \alpha' \leq \beta' \leq \beta$ and $\beta - \alpha < 2 \pi$ and a constructible complex $\cF \in D^b_c (X)$, the natural transform \eqref{eq:!*2} induces a map
\begin{equation}\label{eq:vanindep}
\varphi_u^{[\alpha', \beta']}(\cF) \to \varphi_u^{[\alpha, \beta]}(\cF)
\end{equation}
which is an isomorphism (see \cite[Lemma 1.1.1, Example 1.1.3]{Sch03}).
We define an isomorphism $T_{\pi} \colon \varphi_u \cong \varphi_u^{[3\pi/2, 5\pi/2]}$ by the composition
\[
\varphi_u = \varphi_u^{[\pi/2, 3\pi/2]} \xleftarrow{\sim} \varphi_u^{[3\pi/2, 3\pi/2]} \xrightarrow{\sim }\varphi_u^{[3\pi/2, 5\pi/2]}.
\]
This map is a half of the monodromy operator.
Define closed subsets $X_{\geq 0}$ and $X_{\Re = 0}$ of $X$ by
\[
X_{\geq 0} \coloneqq  \{z \in X \mid \Re (u(z)) \geq 0 \}, \ X_{\Re = 0} \coloneqq  \{z \in X \mid \Re (u(z)) =  0 \}.
\]
We define an isomorphism
\begin{equation}\label{eq:vandual}
\bD_{X_0} \circ \varphi_u \cong \varphi_u \circ \bD_{X}
\end{equation}
 by the following composition:
\begin{align*}
  \bD_{X_0} \circ \varphi_u & = \bD_{X_0} (X_0 \hookrightarrow X_{\leq 0})^* (X_{\leq 0} \hookrightarrow X)^! \\
  &\cong \bD_{X_0} (X_0 \hookrightarrow X_{\Re = 0})^* (X_{\Re = 0} \hookrightarrow X_{\leq 0})^* (X_{\leq 0} \hookrightarrow X)^! \\
  &\cong (X_0 \hookrightarrow X_{\Re = 0})^! (X_{\Re = 0} \hookrightarrow X_{\leq 0})^! (X_{\leq 0} \hookrightarrow X)^* \bD_{X} \\
  &\cong (X_0 \hookrightarrow X_{\Re = 0})^! (X_{\Re = 0} \hookrightarrow X_{\geq 0})^* (X_{\geq 0} \hookrightarrow X)^! \bD_{X} \\
  &\cong (X_0 \hookrightarrow X_{\Re = 0})^* (X_{\Re = 0} \hookrightarrow X_{\geq 0})^* (X_{\geq 0} \hookrightarrow X)^! \bD_{X} \\
  &\cong (X_0 \hookrightarrow  X_{\geq 0})^* (X_{\geq 0} \hookrightarrow X)^! \bD_{X} \\
  &= \varphi_u^{[3\pi/2, 5\pi/2]} \circ \bD_{X} \\
  &\cong \varphi_u \circ \bD_{X}
\end{align*}
where the third isomorphism is the inverse of the base change map \eqref{eq:bc6} which is isomorphic thanks to \cite[Lemma 2.2]{Mas16}, the fourth isomorphism is the natural map \eqref{eq:!*2} which is isomorphic thanks to
\cite[Lemma 2.1]{Mas16}, and the final isomorphism is $T_{\pi}^{-1} \bD_X$.
For a constructible complex $\cF \in D_c^b(X)$ we can show that the following diagram commutes:
  \begin{align}\label{eq:vannatdual}
    \begin{split}
      \xymatrix{
    \bD_{X_0} (X_0 \hookrightarrow X)^* \cF \ar[r] \ar[d]^-{\simd}   \ar[d] & \bD_{X_0} \varphi_u(\cF)  \ar[d]^-{\simd}_-{\eqref{eq:vandual}}   \\
    (X_0 \hookrightarrow X)^! \bD_{X} \cF \ar[r] &    \varphi_u (\bD_{X} \cF)
      }
    \end{split}
  \end{align}
  where the top (resp. bottom) horizontal arrow is the map \eqref{eq:v*} (resp. the map \eqref{eq:!v}).
 This is essentially a consequence of the Verdier self-duality of the natural transform \eqref{eq:!*2}.

\subsection{Fourier--Sato transforms}\label{ssec:FS}

Let $X$ be a topological space with an $\bR^+$-action.
A sheaf $\cF$ on $X$ is called conic if for each $x \in X$, the restriction $\cF |_{\bR^+ \cdot x}$ is locally constant.
We let $D^+_{\bR^+}(X) \subset D^+(X)$ denote the full subcategory consisting of objects whose cohomology sheaves are conic.
Now let $Z$ be a topological space and consider a vector bundle $E$ on $Z$. We equip $E$ with the scaling $\bR^+$-action.
Denote by $\pi_E \colon E \to Z$ the projection and $0_E \colon Z \hookrightarrow E$ the zero section.
Then for an object $\cF \in D^+_{\bR^+}(E)$, it is shown in \cite[Proposition 3.7.5]{KS13} that the following maps
\begin{align}
  {\pi_E}_* \cF \to {\pi_E}_* {0_E}_* 0_E^* \cF \cong 0_E^* \cF \label{eq:homotopyinv1} \\
  0_E^! \cF \cong {\pi_E}_! {0_E}_! 0_E^! \cF \to {\pi_E}_! \cF \label{eq:homotopyinv2}
\end{align}
are isomorphic.

Now we define the Fourier--Sato transform following \cite[\S 3.7]{KS13}.
For applications in this paper, we always work with complex vector bundles.
Let $Z$ be a topological space and $E$ be a complex vector bundle over $Z$.
Denote by $E^{\vee}$ the (complex) dual vector bundle.
Define closed subsets $P, P' \subset E \oplus E^{\vee}$ by
\begin{align*}
P &\coloneqq \{(v, w) \in E \oplus E^{\vee} \mid \Re w(v) \geq 0  \}  \\
P' &\coloneqq \{(v, w) \in E \oplus E^{\vee} \mid \Re w(v) \leq 0  \}.
\end{align*}
Consider the following diagram:
\[
\xymatrix{
{} 
& P' \ar@{^{(}->}[d]_-{\iota_{P'}}
& \\
E
& E \oplus E' \ar[l]_-{p} \ar[r]^-{p'}
& E' \\
{}
& P \ar@{_{(}->}[u]^-{\iota_{P}}
&
}
\]
where $p$ and $p'$ are natural projections and $\iota_P$ and $\iota_{P'}$ are natural inclusions.
We define four functors $\FS_E, \FS_E', \FSS_E, \FSS_E' \colon D^+_{\bR^+}(E) \to \DR^+(E^{\vee})$ by
\begin{align}\label{eq:FSdef}
\begin{split}
  \FS_E \coloneqq p'_! {\iota_{P'}}_* \iota_{P'}^* p^* &,\ \ \ \
  \FS_E' \coloneqq p'_* {\iota_{P}}_! \iota_{P}^! p^* \\
  \FSS_E \coloneqq p'_* {\iota_{P'}}_! \iota_{P'}^! p^! &, \ \  \ \ 
  \FSS_E' \coloneqq p'_! {\iota_{P}}_* \iota_{P}^* p^!.
  \end{split}
\end{align}
The functor $\FS_E$ is called the Fourier--Sato transform.
It is shown in \cite[Theorem 3.7.9]{KS13} that these functors are equivalences.
Now we discuss the relations between these functors.
Let $i_P \colon P \cap P' \hookrightarrow P$ and $i_{P'} \colon P \cap P' \hookrightarrow P'$ the natural inclusions and
consider the following composition
\begin{align*}
\FS_E &= p'_! {\iota_{P'}}_* \iota_{P'}^* p^* \\
      &\xleftarrow{\sim} p'_! {\iota_{P'}}_* {i_{P'}}_! i_{P'}^! \iota_{P'}^* p^* \\
      &\xleftarrow{\sim} p'_! {\iota_{P}}_! {i_{P}}_* i_{P}^* \iota_{P}^! p^* \\
      &\xrightarrow{\sim} p'_* {\iota_{P}}_! {i_{P}}_* i_{P}^* \iota_{P}^! p^* \\
      &\xleftarrow{\sim} p'_* {\iota_{P}}_!  \iota_{P}^! p^* = \FS_E'
\end{align*}
where the first isomorphism follows from \eqref{eq:homotopyinv2},
the second isomorphism follows from \cite[Excercise II.2]{KS13} or \cite[Lemma 2.2]{Mas16},
the third isomorphism follows from \cite[Lemma 3.7.6]{KS13} and the final isomorphism follows from
\eqref{eq:homotopyinv1}.
Similarly, we have a natural isomorphism
\[
\FSS_E \cong \FSS_E'.
\]
Let $a \colon E \xrightarrow{\sim} E$ be the map multiplying $-1$, and for $\cF \in \DR^+(E)$ we write
$\cF^a \coloneqq a^* \cF$.
It is clear that $\FS_E(\cF^a) \cong \FS_E(\cF)^a$.
We write $\FS_E^a \coloneqq \FS_E \circ (-)^a$ and similarly for $\FS'^a_E$, $\FSS_E^a$ and $\FSS'^a_E$.
We have
\begin{align*}
\FSS'^a_E \cong p'_! {\iota_{P'}}_* \iota_{P'}^* p^!
       \cong p'_! {\iota_{P'}}_* \iota_{P'}^* p^*[2 \rank E] = \FS_E[2 \rank E]
\end{align*}
where the second isomorphism is defined  using the orientation on $E$ induced by the complex structure.
Therefore we have isomorphisms
\begin{equation}\label{eq:FSs}
\FS_E' \cong \FS_E  \cong \FSS'^a_E[-2 \rank E] \cong \FSS^a_E[-2 \rank E].
\end{equation}
By definition the functor $\FS_E$ is naturally left adjoint (hence quasi-inverse) to $\FSS_{E^{\vee}}$ and the functor $\FS_E'$ is naturally right adjoint (hence quasi-inverse) to $\FSS_{E^{\vee}}'$ after choosing isomorphisms $p^! \cong p^*[2 \rank E]$ and ${p'}^! \cong {p'}^*[2 \rank E]$ induced by the orientations of $E$ and $E^{\vee}$ coming from the complex structures.
Now consider the following diagram
\[
\xymatrix{
\FSS_{E^{\vee}} \circ \FS_E  \ar[d]_-{\simd} & \ar@{=}[d] \id  \ar[l]_-{\sim} \\
\FSS_{E^{\vee}}' \circ \FS_E' \ar[r]^-{\sim} & \id.
}
\]
Since we always equip complex vector bundles with the orientations induced from the complex structures,
we see that the above diagram commutes up to the sign $(-1)^{\rank E}$ using \cite[Remark 3.7.11]{KS13}.
This implies that the following diagram also commutes up to the sign $(-1)^{\rank E}$:
\begin{equation}\label{eq:FSeta}
  \begin{split}
\xymatrix{
\FSS_{E^{\vee}} \circ \FS_E  \ar[d]_-{\simd} & \ar@{=}[d] \id \ar[l]_-{\sim} \\
(\FS_{E^{\vee}}^a[2 \rank E]) \circ (\FSS_E^a[-2 \rank E]) \ar[r]^-{\sim} & \id.
}
\end{split}
\end{equation}
For $\cF \in \DR^+(E)$, we define an isomorphism
\begin{equation}\label{eq:eta}
\eta(\cF) \colon \FS_{E^{\vee}} \FS_E(\cF) \cong \cF^a[-2 \rank E]
\end{equation}
by the following composition:
\[
\FS_{E^{\vee}} \FS_E(\cF) \cong \FS_{E^{\vee}} \FSS_E^a(\cF) [-2 \rank E] \cong \cF^a[-2 \rank E].
\]
Note that this map differs from the following composition
\[
\FS_{E^{\vee}} \FS_E(\cF) \cong \FSS_{E^{\vee}}^a \FS_E(\cF) [-2 \rank E] \cong \cF^a[-2 \rank E]
\]
by the sign $(-1)^{\rank E}$.
This implies that the following diagram commutes up to the sign $(-1)^{\rank E}$:
\begin{equation}\label{eq:FSetaa}
  \begin{split}
\xymatrix@C=60pt{
\FS_E \FS_{E^\vee} \FS_E(\cF) \ar[r]^-{\FS_E (\eta(\cF))} \ar[d]_-{\eta(\FS_E(\cF))} & \FS_E(\cF^a)[-2 \rank E] \\
\FS_E(\cF^a)[-2 \rank E]. \ar@{=}[ru] &
}
\end{split}
\end{equation}

Now we discuss base change maps for Fourier--Sato transforms.
Let $f \colon Z' \to Z$ be a continuous map between topological spaces and $E$ be a complex vector bundle on $Z$.
We let $E_{Z'}$ denote the base change of $E$ to $Z'$, and $f_{E} \colon E_{Z'} \to E$ and
$f_{E^{\vee}} \colon E_{Z'}^{\vee} \to E^{\vee}$ denote the base changes of $f$.
Using base change maps \eqref{eq:bc1}–\eqref{eq:bc6}, we can construct natural maps
\begin{align}
\FS_E \circ {f_E}_* &\to {f_{E^{\vee}}}_* \circ \FS_{E_{Z'}}, \\
{f_{E^\vee}^*} \circ \FS_E &\to \FS_{E_{Z'}} \circ {f_{E}^*}\label{eq:FSbc^*}, \\
{f_{E^{\vee}}}_! \circ \FS_{E_{Z'}} &\to \FS_E \circ {f_E}_! \label{eq:FSbc^!}, \\
\FS_{E_{Z'}} \circ {f_{E}^!} &\to {f_{E^\vee}^!} \circ \FS_E.
\end{align}
It is shown in \cite[Proposition 3.7.13]{KS13} that all these maps are isomorphic.
We call these maps the base change maps for the Fourier--Sato transform.
We have similar isomorphisms for $\FS'$, $\FSS$, and $\FSS'$, and these isomorphisms commute with the maps in \eqref{eq:FSs}.

Now let $Z$ be a topological space, $E_1$ and $E_2$ be a complex vector bundle over $Z$, and $g \colon E_1 \to E_2$ be a
morphism of vector bundles.
Denote by $^t g\colon E_2^\vee \to E_1^\vee $ the dual map of $g$.
We define $\dim g \coloneqq \rank E_1 - \rank E_2$.
We have an isomorphism
\begin{equation} \label{eq:vectconst}
  g^! \bQ_{E_2} \cong \bQ_{E_1}[2 \dim g]
\end{equation} defined by the following composition
\begin{equation*}
  g^! \bQ_{E_2} \cong g^! \pi_{E_2}^! \bQ_{Z}[-2 \rank E_2] \cong \pi_{E_1}^! \bQ_{Z} [-2 \rank E_2] \cong \bQ_{E_1}[2 \dim g]
\end{equation*}
where the first and third isomorphisms are defined by the orientations on $E_1$ and $E_2$ induced from the complex structures.
For $\cF \in \DR^+(E_1)$ and $\cG \in \DR^+(E_2)$ we will recall the construction of following four isomorphisms defined in \cite[Proposition 3.7.14]{KS13}:
\begin{align}
 \FS_{E_2}(g_! \cF) &\cong {}^t g^* \FS_{E_1}(\cF) \label{eq:FSbcc1}  \\
 \FS_{E_2}(g_* \cF) &\cong {}^t g^! \FS_{E_1}(\cF) [2 \dim g] \label{eq:FSbcc2} \\
 \FS_{E_1}(g^! \cG) &\cong {}^t g_* \FS_{E_2}(\cG) \label{eq:FSbcc3} \\
 \FS_{E_1}(g^* \cG) &\cong {}^t g_! \FS_{E_2}(\cG) [-2 \dim g] \label{eq:FSbcc4}.
\end{align}
Consider the following commutative diagram:
\begin{equation}\label{eq:diagFS}
\begin{split}
\xymatrix{
E_1^\vee & E_2 ^{\vee} \ar[l]_-{^t g} & \\
E_1 \oplus E_1 ^{\vee} \ar[u]^-{p_1'} \ar[rd]_-{p_1} & E_1 \oplus E_2 ^{\vee} \pbcorner \pocorner \ar[l]^-{\widetilde{^t g}} \ar[r]_-{\tilde{g}} \ar[u]_-{\widetilde{p_1'}} \ar[d]_-{\widetilde{p_2}} & E_2 \oplus E_2 ^{\vee} \ar[d]_-{p_2} \ar[lu]_-{p_2'} \\
& E_1 \ar[r]_-{g} & E_2.
}
\end{split}
\end{equation}
Here the upper left and lower right squares are Cartesian.
Now define closed subsets $P_1, P_1' \subset E_1 \oplus E_1^{\vee}$ and  $P_2, P_2' \subset E_2 \oplus E_2^{\vee}$ as before,
and closed subsets $\widetilde{P}, \widetilde{P'} \subset E_1 \oplus E_2 ^{\vee}$ by
\begin{align*}
\widetilde{P} &\coloneqq \{(v, w) \in E_1 \oplus E_2 ^{\vee} \mid \Re w(g(v)) \geq 0 \} \\
\widetilde{P'} &\coloneqq \{(v, w) \in E_1 \oplus E_2 ^{\vee} \mid \Re w(g(v)) \leq 0 \}.
\end{align*}
We let $\iota_{P_1} \colon P_1 \hookrightarrow E_1 \oplus E_1^{\vee}$ denote the inclusion map and define the maps $\iota_{P_1'}, \iota_{P_2}, \iota_{P_2'}, \iota_{\widetilde{P}}, \iota_{\widetilde{P'}}$ in the same way.
The isomorphism \eqref{eq:FSbcc1} is defined by the following composition
\begin{align*}
  \FS_{E_2}(g_! \cF) &= {p'_2}_! {\iota_{P_2'}}_* \iota_{P_2'}^* p_2^* g_! \cF \\
                     &\cong  {p'_2}_! {\iota_{P_2'}}_* \iota_{P_2'}^* \tilde{g}_! \widetilde{p_2}^* \cF \\
                     &\cong  {p'_2}_! \tilde{g}_! {\iota_{\widetilde{P'}}}_* \iota_{\widetilde{P'}}^*  \widetilde{p_2}^* \cF \\
                     &\cong \widetilde{p_1'}_! {\iota_{\widetilde{P'}}}_* \iota_{\widetilde{P'}}^* \widetilde{^t g}^* p_1^* \cF \\
                     &\cong \widetilde{p_1'}_! \widetilde{^t g}^* {\iota_{P_1'}}_* \iota_{P_1'}^*  p_1^* \cF \\
                     &\cong {}^t g^* {p_1'}_! {\iota_{P_1'}}_* \iota_{P_1'}^*  p_1^* \cF  \\
                     &= {}^t g^* \FS_{E_1}(\cF).
\end{align*}
Here all isomorphisms are defined by using the base change maps.
The isomorphism \eqref{eq:FSbcc2} is defined by the following composition
\begin{align*}
  \FS_{E_2}(g_* \cF) \cong \FS_{E_2}'(g_* \cF)
        &= {p'_2}_* {\iota_{P_2}}_! \iota_{P_2}^! p_2^* g_* \cF \\
        &\cong  {p'_2}_* {\iota_{P_2}}_! \iota_{P_2}^! \tilde{g}_* \widetilde{p_2}^* \cF  \\
        &\cong  {p'_2}_* \tilde{g}_* {\iota_{\widetilde{P}}}_! \iota_{\widetilde{P}}^!  \widetilde{p_2}^* \cF  \\
        &\cong \widetilde{p_1'}_* {\iota_{\widetilde{P}}}_! \iota_{\widetilde{P}}^! \widetilde{^t g}^* p_1^* \cF \\
        &\cong \widetilde{p_1'}_* {\iota_{\widetilde{P}}}_! \iota_{\widetilde{P}}^! \widetilde{^t g}^! p_1^* \cF[2 \dim g] \\
        &\cong \widetilde{p_1'}_* \widetilde{^t g}^! {\iota_{P_1}}_! \iota_{P_1}^!  p_1^* \cF[2 \dim g] \\
        &\cong {}^t g^! {p_1'}_* {\iota_{P_1}}_! \iota_{P_1}^!  p_1^* \cF  [2 \dim g]\\
        &= {}^t g^! \FS_{E_1}'(\cF) [2 \dim g] \cong {}^t g^! \FS_{E_1}(\cF) [2 \dim g].
\end{align*}
Here the first and final isomorphisms are defined in \eqref{eq:FSs}, the fifth isomorphism is defined by using \eqref{eq:vectconst} and \eqref{eq:smpurity} and other isomorphisms are defined using the base change maps.
The isomorphism \eqref{eq:FSbcc3} is defined by the following composition
\begin{align*}
  \FS_{E_1}(g^! \cG) &\cong \FS_{E_1} g^! \FS_{E_2^{\vee}} \FS_{E_2} (\cG^a)[2 \rank E_2] \\
                     &\cong \FS_{E_1} \FS_{E_1^{\vee}} {}^t g_* \FS_{E_2} (\cG^a)[2 \rank E_1] \\
                     &\cong  {}^t g_* \FS_{E_2} (\cG)
\end{align*}
where the first and third isomorphisms are defined using $\eta(-)$ in \eqref{eq:eta} and the
second isomorphism is defined using \eqref{eq:FSbcc2}.
Similarly, the isomorphism \eqref{eq:FSbcc4} is defined by the following composition
\begin{align*}
  \FS_{E_1}(g^* \cG) &\cong \FS_{E_1} g^* \FS_{E_2^{\vee}} \FS_{E_2} (\cG^a)[2 \rank E_2] \\
                     &\cong \FS_{E_1} \FS_{E_1^{\vee}} {}^t g_! \FS_{E_2} (\cG^a)[2 \rank E_2] \\
                     &\cong  {}^t g_! \FS_{E_2} (\cG)[-2 \dim g].
\end{align*}
Assume further that we are given a continuous map $f \colon Z' \to Z$.
Denote by $g_{Z'} \colon (E_1)_{Z'} \to (E_2)_{Z'} $ the base change of $g$.
Then for each choice of $a \in \{(-)^*,\ab (-)_*,\ab (-)^!,\ab (-)_!\}$, one can see that the isomorphisms
\eqref{eq:FSbcc1}–\eqref{eq:FSbcc4} commute with base change maps for $a$.
For example if we choose $a = (-)^*$ and \eqref{eq:FSbcc3}, we see that the following diagram commutes:
\begin{equation}\label{eq:FSbccFSbc}
\begin{split}
\xymatrix{
f_{E_1^\vee}^* \FS_{E_1}(g^! \cG) \ar[r]^-{\sim} \ar[d]_-{\simd} & f_{E_1^\vee}^* {}^t g_* \FS_{E_2} (\cG) \ar[d] \\
\FS_{(E_1)_Z}(f_{E_1}^* g^! \cG) \ar[d] &  {}^t {g_{Z'}}_* f_{E_2^\vee}^* \FS_{E_2} (\cG) \ar[d]_-{\simd} \\
\FS_{(E_1)_Z}(g_{Z'}^! f_{E_2}^* \cG) \ar[r]^-{\sim} & {}^t {g_{Z'}}_*  \FS_{(E_2)_Z} (f_{E_2}^* \cG) }.
\end{split}
\end{equation}

We can show that the maps \eqref{eq:FSbcc1} and \eqref{eq:FSbcc2} are associative in the following sense.
Let $E_3$ be another vector bundle over $Z$ and $h \colon E_2 \to E_3$ be morphism of vector bundles.
For $\cF \in \DR^+(E_1)$ ,
the following two diagrams commute:

\begin{equation}\label{eq:FSbcc1ass}
    \xymatrix@C=30pt{
    \FS_{E_3} ((h \circ g)_! \cF) \ar[r]^-{\sim} \ar[dd]^-{\simd}
    & \FS_{E_3} (h_! \circ g_! \cF) \ar[d]^-{\simd} \\
    {}
    & ({}^t h)^* \FS_{E_2} ( g_! \cF) \ar[d]^-{\simd} \\
    {({}^t g \circ {}^t h)^* \FS_{E_1}(\cF)} \ar[r]^-{\sim}
    & ({}^t h)^* ({}^t g)^* \FS_{E_1} ( \cF)
    }
\end{equation}
\begin{equation}\label{eq:FSbcc2ass}
                 \xymatrix@C=30pt{
    \FS_{E_3} ((h \circ g)_* \cF) \ar[r]^-{\sim} \ar[dd]^-{\simd}
    & \FS_{E_3} (h_* \circ g_* \cF) \ar[d]^-{\simd} \\
    {}
    & ({}^t h)^! \FS_{E_2} ( g_* \cF)[2 \dim h] \ar[d]^-{\simd} \\
    {({}^t g \circ {}^t h)^! \FS_{E_1}(\cF)}[2 \dim g + 2 \dim h] \ar[r]^-{\sim}
    & ({}^t h)^! ({}^t g)^! \FS_{E_1} ( \cF)[2 \dim g + 2 \dim h].
    }
\end{equation}
We omit the proof as it is a routine argument in sheaf theory.
We have similar statements for maps \eqref{eq:FSbcc3} and \eqref{eq:FSbcc4}.

Now we discuss the compatibility between Fourier--Sato transforms and natural transformations discussed in \S\ref{sect:nattrans}.
\begin{prop}\label{prop:FSpure}
  The following diagram commutes:
  \[
  \xymatrix{
  \FS_{E_2}(g_! \cF)\ar[d]^-{\simd} \ar[r] & \FS_{E_2}(g_* \cF) \ar[d]^-{\simd} \\
  {}^t g^*\FS_{E_1}(\cF) \ar[r] & {}^t g^!\FS_{E_1}(\cF)[-2 \dim ^t g].
  }
  \]
  Here the upper horizontal arrow is the map \eqref{eq:!*1} and the lower horizontal arrow is defined by using \eqref{eq:vectconst} and \eqref{eq:smpurity}.
\end{prop}
The proof of this proposition will be given in \S \ref{ssec:FSpure}.

\begin{prop}\label{prop:FSunit}
  Consider following diagrams
  \[
  \xymatrix{
  \FS_{E_2}(\cG) \ar[r] \ar@/_18pt/[rdd] & \FS_{E_2}(g_* g^* \cG) \ar[d]^-{\simd} \\
  & {}^t g^! \FS_{E_1}( g^* \cG)[2 \dim g] \ar[d]^-{\simd} \\
  & {}^t g^! \, {}^t g_! \FS_{E_2}( \cG)
  } \hspace{20pt}
  \xymatrix{
  \FS_{E_2}(g_! g^! \cG) \ar[d]^-{\simd} \ar[r]  & \FS_{E_2}( \cG) \\
  {}^t g^*\FS_{E_1}( g^! \cG) \ar[d]^-{\simd} &  \\
  {}^t g^* \, {}^t g_* \FS_{E_2}( \cG). \ar@/_18pt/[ruu] &
  }
  \]
  Then the left diagram commutes up to the sign $(-1)^{\rank E_1}$ and the right diagram commutes up to the sign $(-1)^{\rank E_2}$.
\end{prop}
The proof of this proposition will be given in \S \ref{ssec:FSunit}.


\section{Specialization functor and specialization map}\label{sec:sp}

Let $X$ be a separated scheme of finite type over $\Spec \bC$ and $Z$ be a closed subscheme of $X$.
We let $C_{Z/X}$ denote the normal cone of $Z$ in $X$.
In this section, we introduce the specialization functor
\[
\Sp_{Z/X} : D^+(X) \to D^+(C_{Z/X})
\] following Verdier \cite{Ver81},
and define the specialization map for the Borel--Moore homology groups
\[
\Sp^{\BM}_{Z/X} \colon \HBM_*(X) \to \HBM_*(C_{Z/X})
\]
using the specialization functor.
Then we prove that the specialization map $\Sp^{\BM}_{Z/X}$ is compatible with the specialization map
of Chow groups defined in Fulton's book \cite[\S 5]{Ful84}.

\subsection{Specialization functors}

Let us recall the definition and basic properties of the specialization functor introduced by Verdier \cite{Ver81}.
Let $X$ be a separated scheme of finite type over $\Spec \bC$, and $Z$ be a closed subscheme.
Define schemes $M_{Z/X}$ and $M_{Z/X}^0$ by
\begin{align*}
M_{Z/X} &\coloneqq \Bl_{Z \times \{0\}}(X \times \bA^1) \\
M_{Z/X}^0 &\coloneqq M_{Z/X} \setminus \Bl_{Z}(X).
\end{align*}
Then we have the following commutative diagram:
\begin{equation}\label{eq:spdiag}
\begin{tikzcd}
Z \arrow[r, hookrightarrow] \arrow[d, hookrightarrow] & Z \times \bA^1 \arrow[r, hookleftarrow] \arrow[d, hookrightarrow, "i_{Z \times \bA^1}"] & Z\times (\bA^1 \setminus \{0\}) \arrow[d, hookrightarrow] &\\
C_{Z/X} \arrow[r, hookrightarrow, "i_{C_{Z/X}}"] \arrow[d] & M_{Z/X}^0 \arrow[r, hookleftarrow, "j"] \arrow[d, "p"] & X \times (\bA^1 \setminus \{0\}) \arrow[d] \arrow[r,  "\pr_1"] & X \\
\{ 0 \} \arrow[r, hookrightarrow] & \bA^1 \arrow[r, hookleftarrow] & (\bA^1 \setminus \{0\}) &
\end{tikzcd}
\end{equation}

\begin{dfn}[{{\cite[\S 8]{Ver81}}}]\label{def:specialization}
  The specialization functor
  \[
  \Sp_{Z/X} \colon D^+(X) \to \DR^+(C_{Z/X})
  \] is defined by the following formula:
    \[
    \Sp_{Z/X}(\cF) \coloneqq \psi_{p} j_! \pr_1^*(\cF).
    \]
  Here $\bR^+$-action on $C_{Z/X}$ is defined by the scaling of fibers.
\end{dfn}

The followings are basic properties of the specialization functor we use in this paper.

\begin{prop}[{\cite[\S9]{Ver81}}]\label{prop:spres}
  Let $X$ be a separated scheme of finite type over $\Spec \bC$ and $Z \subset X$ be a closed subscheme.
  For a constructible complex $\cF \in D_c^b(X)$, we have following isomorphisms
  \begin{align}
    \Sp_{Z/X}(\cF) |_Z &\cong \cF |_Z, \label{eq:spres*} \\
    \Sp_{Z/X}(\cF) |_Z^! &\cong \cF |_Z^! \label{eq:spres!} .
  \end{align}
\end{prop}

\begin{proof}
    Firstly note that we have natural morphisms
    \[
    \Sp_{Z/X} (\cF)|_Z \to \Sp_{Z/X}(\cF_Z)|_Z, \ \ \Sp_{Z/X}(\Gamma_Z(\cF))|^!_Z \to \Sp_{Z/X}(\cF)|^!_Z.
    \]
    We claim that these maps are isomorphisms.
    To do this, it is enough to prove that for complexes $\cF \in D^b_c(X)$ with $\cF |_Z = 0$ we have $\Sp_{Z/X}(\cF) |_Z = 0$,
    and for $\cF \in D^b_c(X)$ with $\cF |_Z^! = 0$ we have $\Sp_{Z/X}(\cF) |_Z^! = 0$.
    By shrinking $X$ and embedding it into an affine space so that $Z$ is cut out by linear equations, we may assume $X$ and $Z$ are smooth.
    Then these claims follow from \cite[Theorem 4.2.3]{KS13}.
    
    Now consider the composition
    \begin{equation}\label{eq:sp*}
    \Sp_{Z/X}(\cF_Z) |_Z \to \psi_{p \circ i_{Z \times \bA^1}} ((Z \times \bA^1 \to Z)^* (\cF |_Z)) \to \cF |_Z
  \end{equation}
    where the first map is \eqref{eq:nvbc1} and the second map is \eqref{eq:v*}. It is clear that these maps are isomorphisms, hence we obtain the isomorphism \eqref{eq:spres*}.
    Similarly, we have isomorphisms
      \begin{equation}\label{eq:sp!}
    \Sp_{Z/X}(\Gamma_Z(\cF)) |^! _Z \cong \psi_{p \circ i_{Z \times \bA^1}} ((Z \times \bA^1 \to Z)^! (\cF |_Z^![-2])) \cong \cF |_Z^!
    \end{equation}
    which imply the isomorphism \eqref{eq:spres!}.

\end{proof}

\subsection{Specialization maps}\label{ssec:spmap}

Consider the natural map defined in \eqref{eq:purity}
\[
\Sp_{Z/X}(\bQ_X) |_Z \otimes \omega_{Z/ C_{Z/X}} \to \Sp_{Z/X}(\bQ_X) |_Z^!.
\]
Using the proposition above, this map is identified with a map
\begin{equation}\label{eq:spmapp}
\omega_{Z/ C_{Z/X}} \to \omega_{Z/ X}.
\end{equation}
The isomorphism \eqref{eq:homotopyinv2} implies $\omega_{Z/ C_{Z/X}} \cong (C_{Z/X} \to Z)_! \bQ_{C_{Z/X}}$ hence by adjunction we obtain an element
\[
  e_{Z/X} \in \mH^0(C_{Z/X} \to X).
\]
Define a map
\[
\Sp_{Z/X}^{\BM} \colon \HBM_*(X) \to \HBM_*(C_{Z/X})
\]
by composing $e_{Z/X}$ to elements in $\HBM_*(X) = \mH^{-*}(X \to \Spec \bC)$.
Now we compare this specialization map for Borel--Moore homology groups and the specialization map for Chow groups
\[
\Sp_{Z/X}^{\Chow} \colon A_*(X) \to A_*(C_{Z/X})
\]
defined in Fulton's book \cite[\S 5]{Ful84}.
\begin{thm}\label{thm:spcomm}
  The following diagram commutes:
  \[
  \xymatrix@C=50pt{
  A_*(X) \ar[r]^-{\Sp_{Z/X}^{\Chow}} \ar[d]^-{\cl_{X}} & A_*(C_{Z/X}) \ar[d]^-{\cl_{C_{Z/X}}} \\
  \HBM_*(X) \ar[r]^-{\Sp_{Z/X}^{\BM}}  & \HBM_*(C_{Z/X})
  }
  \]
  where $\cl_X$ and $\cl_{C_{Z/X}}$ are cycle maps.
\end{thm}

This theorem will be proved in two steps:
firstly we construct the Gysin pullback for Borel--Moore homology to a principal divisor using the nearby cycle functor and compare it with the Gysin pullback for Chow groups (Proposition \ref{prop:clGysin}), and then reduce the theorem to this statement.

Let $X$ be a separated scheme of finite type over $\Spec \bC$ and $u \colon X \to \bA^1$ be a regular function.
Write $D \coloneqq u^{-1}(0)$ and we let $i \colon D \hookrightarrow X$ denote the inclusion map.
The composition of morphisms in \eqref{eq:nvdist1} and \eqref{eq:nvdist2} defines the following map
\[
(D \hookrightarrow X)^* \bQ_X \to \psi_u(\bQ_X) \to (D \hookrightarrow X)^! \bQ_X[2]
\]
hence an element $e_u \in \mH^2(D \to X)$.
The element $e_u$ defines a map
\[
 _\psi i^! \colon \HBM_*(X) \to \HBM_{* - 2}(D).
\]
Now let $f \colon X' \to X$ be a proper map and write $D' \coloneqq (u \circ f)^{-1} (0)$. 
We let $i' \colon D' \hookrightarrow X'$ denote the inclusion map.
\begin{lem}
The following diagram commutes:
\begin{align}\label{eq:Gysinpush}
  \begin{split}
  \xymatrix{
  \HBM_*(X') \ar[r]^-{_{\psi} {i'}^!} \ar[d]^{f_*}  &   \HBM_{*-2}(D') \ar[d]^{(f|_{D'})_*}  \\
  \HBM_*(X) \ar[r]^-{_{\psi} i^!} & \HBM_{*-2}(D).
  }
\end{split}
\end{align}
\end{lem}

\begin{proof}
  Consider the following diagram in $D^b(D)$:
  \[
  \xymatrix{
  \bQ_{D} \ar[r]^-{\sim} \ar[dd] \ar@{}[ddr] | {\textstyle{\text{(A)}}}
  & i^* \bQ_{X} \ar[r] \ar[d] \ar@{}[dr] | {\textstyle{\text{(B)}}}
  & \psi_u(\bQ_X) \ar[r] \ar[d] \ar@{}[dr] | {\textstyle{\text{(C)}}}
  & i^! \bQ_{X}[2] \ar[d] \\
  {}
  & i^* f_* \bQ_{X'} \ar[r] \ar[d]^-{\simd} \ar@{}[dr] | {\textstyle{\text{(D)}}}
  & \psi_u(f_* \bQ_{X'}) \ar[r] \ar[d]^-{\simd} \ar@{}[dr] | {\textstyle{\text{(E)}}}
  & i^! f_* \bQ_{X'}[2] \ar[d]^-{\simd} \\
  (f|_{D'})_* \bQ_{D'} \ar[r]^-{\sim}
  & (f|_{D'})_* {i'}^* \bQ_{X'} \ar[r]
  &(f|_{D'})_* \psi_{u \circ f} (\bQ_{X'}) \ar[r]
  & (f|_{D'})_* {i'}^! \bQ_{X'}[2].
  }
  \]
  The commutativity of diagrams (A), (B), and (C) is obvious.
  The commutativity of the diagrams (D) and (E) follows from the fact that 
  the maps in \eqref{eq:nvdist1} and \eqref{eq:nvdist2} commute with the base change maps.
  By taking the Verdier-dual of the outer square of the above diagram, we obtain the desired claim.
\end{proof}


\begin{prop}\label{prop:clGysin}
  Assume that $D \subset X$ is a Cartier divisor (i.e $f^{\red}$ is not identically zero over any irreducible component of $X$).
  Then the following diagram commutes:
  \[
  \xymatrix@C=50pt{
  A_*(X) \ar[r]^-{i^!} \ar[d]^-{\cl_{X}} & A_*(D) \ar[d]^-{\cl_{D}} \\
  \HBM_*(X) \ar[r]^-{_{\psi} i^!}  & \HBM_*(D)
  }
  \]
  where $i^!$ is the usual Gysin pullback for Chow groups.
\end{prop}

\begin{proof}
  We need to prove
  \[
  \cl_D \circ i^! ([V]) = { _{\psi}} i^! ([V]_{\BM})
  \]
  for each irreducible subvariety $V \subset X$.
  Firstly consider the case when $V$ is included in $D$.
  Then $i^! [V] = 0$ since $D$ is linearly equivalent to $0$,
  and the commutativity of \eqref{eq:Gysinpush} implies ${ _{\psi}} i^! ([V]_{\BM})= 0$.

  Now we may assume $V$ and $D$ intersect properly.
  Replacing $X$ by $V$ using the commutativity of \eqref{eq:Gysinpush},
  we may assume $X = V$.
  Further replacing $X$ by its normalization and using the commutativity of \eqref{eq:Gysinpush} again and shrinking $X$ if necessary, we may assume
  $X$ is smooth and $D$ is irreducible.
  In this case $ i^! ([V]) = d[D^{\red}]$ where $d$ is the multiplicity of $D$ in $X$.
  Now we want to compute
  \[
  { _{\psi}} i^!  ([X]_{\BM}) \in \HBM_{2 \dim X -2} (D) = \bZ \cdot [D^{\red}]_{\BM}.
  \]
  Since this class can be computed analytically locally, we may replace $X$ by $\bA^{n}$ and $u$ by $z_1^d$ where $z_1$ is the first projection.
  Using the commutativity of \eqref{eq:Gysinpush} for the map $f = (z_1^d, \id_{\bA^{n-1}})$, we may assume $d = 1$.
  In this case, the claim follows from Proposition \ref{prop:nearbyid}.
\end{proof}

\begin{proof}[Proof of Theorem \ref{thm:spcomm}]
We use notations from the diagram \eqref{eq:spdiag}.
We need to prove
\[
\cl_{C_{Z/X}} (\Sp^{\Chow}_{Z/X}([V])) = \Sp^{\BM}_{Z/X}([V]_{\BM})
\]
for each irreducible subvariety $V \subset X$ of dimension $d$.
We let $\widetilde{V} \subset M_{Z/X}^o$ denote the strict transform of $V \times \bA^1$.
Proposition \ref{prop:clGysin} claims the equality
\[
\cl_{C_{Z/X}} (\Sp^{\Chow}_{Z/X}([V])) = ({_{\psi}}i_{C_{Z/X}})^!([\widetilde{V}]_{\BM})
\]
hence what we need to prove is the equality
\[
\Sp^{\BM}_{Z/X}([V]_{\BM}) = ({_{\psi}}i_{C_{Z/X}})^!([\widetilde{V}]_{\BM}).
\]

Consider the following diagram:
\begin{equation}\label{eq:diag1}
\xymatrix@C=35pt{
 \bQ_{C_{Z/X}} \ar[r] \ar@{=}[d] & \psi_{p}(\bQ_{M_{Z/X}^o}) \ar[r]^-{\psi_p(- \cap [\widetilde{V}]_{\BM})} \ar[d]^-{\simd} & \psi_{p}(\omega_{M_{Z/X}^o})[-2d-2] \ar[r] \ar[d]^-{\simd} & \omega_{C_{Z/X}}[-2d] \ar@{=}[d]  \\
 \bQ_{C_{Z/X}} \ar[r] & \Sp_{Z/X}(\bQ_X) \ar[r]^-{\Sp_{Z/X}(- \cap [V]_{\BM})} & \Sp_{Z/X}(\omega_X)[-2d] \ar[r] & \omega_{C_{Z/X}}[-2d].
}
\end{equation}
Here the middle left and middle right vertical arrows are defined by isomorphisms $\pr_1^* \bQ_{X} \cong \bQ_{X \times (\bA^1  \setminus\{0 \})}$ and $\pr_1^* \omega_{X} \cong \omega_{X \times (\bA^1  \setminus\{0 \})}[-2]$ respectively,
upper left and upper right horizontal arrows are defined in \eqref{eq:nvdist1} and \eqref{eq:nvdist2} respectively, and
lower left and lower right horizontal arrows are defined so that left and right squares commute.
The commutativity of the middle diagram follows from the equality
\[
\pr_1^* [V] = [\widetilde{V} \cap (X \times  (\bA^1  \setminus\{0 \}))].
\]
The composition of upper horizontal arrows corresponds to the element
\[
({_{\psi}}i_{C_{Z/X}})^!([\widetilde{V}]_{\BM}) \in \HBM_{2d}(C_{Z/X})
\] by definition.
Therefore we need to prove that the composition of lower horizontal arrows corresponds to the element $\Sp^{\BM}_{Z/X}([V]_{\BM})$.
By the definition of $\Sp^{\BM}_{Z/X}$, this statement is equivalent to the commutativity of the outer rectangle of the following diagram
\[
\xymatrix@C=10pt{
 \omega_{Z/C_{Z/X}}  \ar[rr] \ar@{=}[d] & & \Sp_{Z/X}(\bQ_X) |^!_Z  \ar[rr]  & & \Sp_{Z/X}(\omega_X)[-2d]|_{Z}^! \ar[rr]^-{\sim} & & \ar@{=}[d] \omega_Z [-2d] \\
 \omega_{Z/C_{Z/X}} \ar[rrr]^-{\eqref{eq:spmapp}} & & & \omega_{Z/X} \ar[rrr]^-{(Z \hookrightarrow X)^!(- \cap [V]) } \ar@{-->}[lu]^-{\sim}_-{\eqref{eq:spres!}} & & & \omega_{Z}[-2d]
}
\]
where the upper horizontal arrows are obtained by applying the functor $(Z \hookrightarrow C_{Z/X})^!$ to the maps in the lower horizontal arrows of the diagram \eqref{eq:diag1}.
The commutativity of the right quadrilateral is obvious.
Note that the composition of \eqref{eq:spres!} and \eqref{eq:spmapp} is obtained by applying the functor $(Z \hookrightarrow C_{Z/X})^!$ to
the composition
\[
\bQ_{C_{Z/X}} \otimes (Z \hookrightarrow C_{Z/X})_* \omega_{Z/C_{Z/X}} \xrightarrow[\sim]{\eqref{eq:spres*}} \Sp_{Z/X}(\bQ_X) \otimes (Z \hookrightarrow C_{Z/X})_* \omega_{Z/C_{Z/X}} \to \Sp_{Z/X}(\bQ_X)
\]
where the second map is the $!$-counit map.
Then the commutativity of the left quadrilateral follows from the construction of the map \eqref{eq:spres*}.

\end{proof}


\section{Fourier--Sato transform of the specialization}\label{sec:FS}

The aim of this section is to prove the following statement:

\begin{thm}\label{thm:vansp}
 Let $X$ be a separated scheme of finite type over $\Spec \bC$, $E$ be a vector bundle on $X$, and $s \in \Gamma(X, E)$ be a section.
 Write $Z = s^{-1}(0)$ and $E_Z = E |_{Z}$. Denote by $\pi_{E^\vee} \colon E^\vee \to X$ the projection and $\bar{s} \colon E^\vee \to \bA^1$ the regular function corresponding to $s$.
 We let $\iota_{C_{Z/X}} \colon C_{Z/X} \hookrightarrow E_Z$ denote the natural inclusion from the normal cone.
 Then for a constructible complex $\cF \in D_c^b(X)$, we have an isomorphism
 \begin{align}\label{eq:vansp}
 \FS_{E_Z^\vee}(\vphi_{\bar{s}}(\pi_{E^\vee}^!\cF)|_{E_Z^\vee})  \cong \iota_{C_{Z/X},*} \Sp_{Z/X}(\cF).
\end{align}

\end{thm}
As a consequence, we will give a new proof of Davison's dimensional reduction theorem \cite[Theorem A.1]{Dav17} and prove our main result for local models for quasi-smooth derived schemes.

\subsection{Proof of Theorem \ref{thm:vansp}}

Firstly, we rewrite the complex $\vphi_{\bar{s}}(\pi_{E^\vee}^!\cF)|_{E_Z^\vee}$ so that its Fourier--Sato transform can be easily computed.
Let $\ell_{\leq 0} \subset \bC$ denote the closed half line consisting of the non-positive real numbers.
Consider the following commutative diagram:
\begin{align}\label{eq:diag3.1}
  \begin{split}
\xymatrix@C=40pt{
E_{Z}^\vee \ar@{^{(}->}[r]^-{i_{E_Z^\vee}} \ar[d]^-{\pi_{E^\vee_Z}} & E^\vee \ar[d]^-{\pi_{E^\vee}} \ar[r]^-{\hat{s}} & X \times \bA^1 \ar[ld]^(.3){p_{\bA^1}} & X \times \ell_{\leq 0} \ar[lld]^-{p_{\ell_{ \leq 0}}} \ar@{_{(}->}[l]_-{i_{X \times \ell_{\leq 0}}} \\
 Z \ar@{^{(}->}[r]^-{i_Z} & X & &
}
\end{split}
\end{align}
where we define $\hat{s} \coloneqq  (\pi_{E^\vee}, \bar{s})$, $i_{E_Z^\vee}$, $i_Z$ and $i_{X \times \ell_{\leq 0}}$ are natural inclusions, and $\pi_{E^\vee_Z}$, $\pi_{E^\vee}$, $p_{\bA^1}$ and $p_{\ell_{ \leq 0}}$
are natural projections.
Consider the following complex
\[
\vphi_{\bar{s}}^{\ell_{\leq 0}}(\pi_{E^\vee}^!\cF) \coloneqq (\bar{s}^{-1}(0) \hookrightarrow \bar{s}^{-1}(\ell_{\leq 0}))^*(\bar{s}^{-1}(\ell_{\leq 0}) \hookrightarrow E^\vee)^!\pi_{E^\vee}^!\cF.
\]
As we have seen in \eqref{eq:vanindep}, we have a natural isomorphism
\begin{equation*}
\vphi_{\bar{s}}^{\ell_{\leq 0}}(\pi_{E^\vee}^!\cF) \xrightarrow[]{\cong} \vphi_{\bar{s}}(\pi_{E^\vee}^!\cF).
\end{equation*}
Therefore we have
\begin{align}\label{eq:revan}
  \vphi_{\bar{s}}(\pi_{E^\vee}^!\cF) |_{E_Z^\vee} \cong
    (\hat{s}^!  i_{X \times \ell_{\leq 0}, *} p_{\ell_ {\leq 0}}^! \cF) |_{E_Z^\vee}.
\end{align}
Applying the functor $\FS_{E^\vee_Z}$ to both sides, we obtain
\begin{align}\label{eq:FSrevan}
  \begin{split}
    \FS_{E^\vee_Z}(\vphi_{\bar{s}}(\pi_{E^\vee}^!\cF) |_{E_Z^\vee})
    & \cong \FS_{E^\vee_Z}( (\hat{s}^!  i_{X \times \ell_{\leq 0}, *}  p_{\ell_ {\leq 0}}^! \cF) |_{E_Z^\vee}) \\
    & \cong \FS_{E^\vee}( \hat{s}^!  i_{X \times \ell_{\leq 0}, *}  p_{\ell_ {\leq 0}}^! \cF) |_{E_Z} \\
    & \cong (\tilde{s}_* \FS_{X \times \bA^1}(  i_{X \times \ell_{\leq 0}, *}  p_{\ell_ {\leq 0}}^! \cF)) |_{E_Z}
 \end{split}
\end{align}
where $\tilde{s}$ is the dual map of $\hat{s}$.
The second and third isomorphism are defined in \eqref{eq:FSbc^*} and \eqref{eq:FSbcc3} respectively.

Consider the following diagram:
\begin{align*}
  \xymatrix@C=40pt{
  E_Z \ar@{^{(}->}[r]^-{i_{E_Z}} \ar[d]^-{\pi_{E_Z}} & E \ar[d]^-{\pi_E}  & X \times \bA^1 \ar[l]_-{\tilde{s}} \ar[ld]^(.3){p_{\bA^1}} & X \times {\bA^1_{>0}} \ar[l]_-{j_{X \times \bA^1_{>0}}} \ar[lld]^-{p_{\bA^1_{>0}}} \\
 Z \ar@{^{(}->}[r]^-{i_Z}  & X & &
  }
\end{align*}
where $\bA^1_{>0} \subset \bA^1$ is the open subset consisting of points whose real part is positive, $i_{E_Z}$ and $j_{X \times \bA^1_{>0}}$ are natural inclusions and $\pi_E$, $\pi_{E_Z}$, $p_{\bA^1}$ and $p_{\bA^1_{>0}}$ are natural projections.

\begin{lem}\label{lem:FShalf}
  There exists an isomorphism
  \begin{equation}\label{eq:FShalf}
  \FS_{X \times \bA^1}(  i_{X \times \ell_{\leq 0}, *}  p_{\ell_ {\leq 0}}^! \cF))
  \cong j_{X \times \bA^1_{>0}, *}  p_{\bA^1_ {>0}}^* \cF.
\end{equation}
\end{lem}

\begin{proof}

Let $\cG$ be the complex at the left-hand side of the claim.
For a point $\xi \in \bA^1 \setminus \bA^1_{>0}$,
we claim the vanishing $\cG |^!_{X \times \{ \xi \}} = 0$.
Using the isomorphism \eqref{eq:FSbc^!}, we may assume that $X$ is a point which will be denoted by $\pt$ and $\cF$ is a constant sheaf $\bQ_{\pt}$.
Now recall that we have seen in \eqref{eq:FSs} that there exists an isomorphism of functors $\FS_{\bA^1} \cong \FS'_{\bA^1}$ where the functor  $\FS'_{ \bA^1}$ is defined in \eqref{eq:FSdef}.
We define closed subsets $H_{\xi}, H_{\xi}'  \subset \bA^1$ by
\begin{align*}
H_{\xi} \coloneqq \{ t \in \bA^1 \mid \Re (t \cdot \xi) \geq 0 \}, \ \ 
H_{\xi}' \coloneqq \{ t \in \bA^1 \mid \Re (t \cdot \xi) \leq 0 \}.
\end{align*}
By the base change theorem, we have
\[
    \FS'_{ \bA^1}(  i_{ \ell_{\leq 0}, *}  p_{\ell_ {\leq 0}}^! \bQ_{\pt})|^!_{\xi} 
    \cong R \Gamma(H_{\xi}, (i_{ \ell_{\leq 0}, *}  p_{\ell_ {\leq 0}}^! \bQ_{\pt})|_{H_{\xi}}^!)
\]
where $i_{ \ell_{\leq 0}} \colon  \ell_{\leq 0} \hookrightarrow \bA^1$ is the natural inclusion.
The assumption on $\xi$ implies that $H_{\xi}$ contains $\ell_{\leq 0}$ and it is easy to see that the complex at the right-hand side vanishes.
Therefore the unit map
\begin{align}\label{eq:unitisom}
\cG \to j_{X \times \bA^1_{>0}, *} (\cG |_{X \times \bA^1_{>0}})
\end{align}
is isomorphic.

For $\xi \in \bA^1_{> 0}$, the half line $l_{\leq 0}$ is contained in $H_{\xi}'$.
This implies the isomorphism
\[
\cG |_{X \times \bA^1_{>0}} \cong ({\pr_{13}}_! \pr_{12}^* i_{X \times \ell_{\leq 0}, *} p_{ \ell_{\leq 0}}^! \cF) |_{X \times \bA^1_{>0}}.
\]
By the base change theorem, we see that the right-hand side is isomorphic to $p_{\bA^1 > 0}^* \cF$. This isomorphism and the isomorphism \eqref{eq:unitisom} imply the desired result.
\end{proof}

Combining \eqref{eq:FSrevan} and Lemma \ref{lem:FShalf}, the complex at the left-hand side of \eqref{eq:vansp} is isomorphic to the complex
\begin{equation}\label{eq:inter1}
(\tilde{s}_* j_{X \times \bA^1_{>0}, *}  p_{\bA^1_{>0}}^* \cF)|_{E_Z}.
\end{equation}
Consider the following diagram
\begin{equation}
  \begin{split}
  \xymatrix@C=45pt{
X & \ar[l]_-{p_{\bA^1_{>0}}}  X\times \bA^1_{>0} \ar[r]^-{(s \circ p_{\bA^1_{>0}}) \cdot (\tau \circ \pr_2)} \ar[d]_-{((s \circ p_{\bA^1_{>0}}) \cdot (\tau \circ \pr_2), j_{\bA^1_{>0}} \circ \pr_2)} & E \ar@{^{(}->}[ld]^-{0_{E/E \times \bA^1}}  & E_Z \ar@{_{(}->}[l]_-{i_{E_Z}} \\
& E \times \bA^1 & &
  }
\end{split}
\end{equation}
where $\tau \colon \bA^1_{>0} \xrightarrow{\sim} \bA^1_{>0}$ is the involution taking the inverse,
$j_{\bA^1_{>0}} \colon \bA^1_{>0} \hookrightarrow \bA^1$ is the natural open inclusion, $0_{E/E \times \bA^1} \colon E \hookrightarrow E \times \bA^1$ is the zero section,
and $\pr_2 \colon X \times \bA^1_{>0} \to \bA^1_{>0}$ is the projection to the second factor.
To simplify the notation, we write
 \begin{align*}
 \tilde{s}_{\tau} &\coloneqq (s \circ p_{\bA^1_{>0}}) \cdot (\tau \circ \pr_2) \colon X\times \bA^1_{>0} \to E \\
 \tilde{s}_{\tau, p} &\coloneqq (\tilde{s}_{\tau}, j_{\bA^1_{>0}} \circ \pr_2) \colon X\times \bA^1_{>0} \to E \times \bA^1
\end{align*}
Note that there exists a natural isomorphism
\begin{align}\label{eq:nattau}
(\tilde{s}_* j_{X \times \bA^1_{>0}, *}  p_{\bA^1_{>0}}^* \cF)|_{E_Z} \cong  (\tilde{s}_{\tau, *}  p_{\bA^1_{>0}}^*\cF) |_{E_Z}.
\end{align}
Let $\pi_{E \times \bA^1/E} \colon E\times \bA^1 \to E$ be the projection.
Then the natural transformation
\[
\pi_{E \times \bA^1/E, *} 
\to  \pi_{E \times \bA^1/E, *} 0_{E/E \times \bA^1, *} 0_{E/E \times \bA^1}^* 
  \xrightarrow{\sim} 0_{E/E \times \bA^1}^*
\]
induces a map of complexes on $E_Z$
\begin{align}\label{eq:inter3}
 (\tilde{s}_{\tau, *}  p_{\bA^1_{>0}}^* \cF)|_{E_Z}
\to   (0_{E/E \times \bA^1}^*  \tilde{s}_{\tau, p, *}  p_{\bA^1_{>0}}^* \cF)|_{E_Z} .
\end{align}

\begin{lem}\label{lem:isoun}
  The map \eqref{eq:inter3} is an isomorphism.
\end{lem}
\begin{proof}
  We let $\cL$ (resp. $\cR$) denote the complex at the left-hand side (resp. right-hand side) of the statement.
  We need to prove the induced map
  \begin{align}\label{eq:lrp}
  \mH^i (\cL_p) \to \mH^i (\cR_p)
\end{align}
  is isomorphic for each $i \in \bZ$ and $p \in E_Z$. We have isomorphisms
  \begin{align*}
  \mH^i (\cL_p) &\cong \varinjlim_{p \in V \subset E} \mH^i(\tilde{s}_{\tau}^{-1}(V), p_{\bA^1_{>0}}^* \cF) \\
  \mH^i (\cR_p) &\cong \varinjlim_{{\substack{p \in V \subset E \\ 0 \in D \subset \bA^1}}} \mH^i(\tilde{s}_{\tau}^{-1}(V) \cap (X \times D), p_{\bA^1_{>0}}^* \cF)
\end{align*}
where $V$ ranges over all open subsets of $E$ containing $p$ and $D$ ranges over all open subsets of $\bA^1$ containing the origin.

  Firstly assume that $p$ is not contained in the zero section.
  We will show that for each choice of $(V, D)$ as above, there exists an open subset $V' \subset E$ containing $p$ such that the inclusion
  \[
  \tilde{s}_{\tau}^{-1}(V') \subset \tilde{s}_{\tau}^{-1}(V) \cap (X \times D)
  \]
  holds, which implies that the map \eqref{eq:lrp} is an isomorphism.
  To see this, note that the map $\tilde{s}_{\tau} |_{X \times (\bA^1_{>0} \setminus D)}$ naturally extends to a map
  \[
  \tilde{s}_{\tau, \bP^1 \setminus D} \colon X \times (\bP^1 \setminus D) \to E.
  \]
  which is proper.
  It is clear that $p$ is not contained in the image of $\tilde{s}_{\tau, \bP^1 \setminus D}$ since we assumed $p \in E_Z \setminus Z$.
  Then we can take $V'$ as $V \setminus \mathrm{Im}( \tilde{s}_{\tau, \bP^1 \setminus D})$.

  Now suppose that $p$ is contained in the zero section.
  Let $V$ be a convex open neighborhood of $p$ in $E$ and $D$ be a convex open neighborhood of $0$ in $\bA^1$.
  For each point $q \in X \times \bA^1$, it is clear that $(\bR^+ \cdot q) \cap  \tilde{s}_{\tau}^{-1}(V)$ is contractible. Therefore
    \cite[Corollary 3.7.3]{KS13} implies an isomorphism
  \begin{align*}
    \mH^i(\tilde{s}_{\tau}^{-1}(V), p_{\bA^1_{>0}}^* \cF) \cong \mH^i(p_{\bA^1_{>0}}(\tilde{s}_{\tau}^{-1}(V)),  \cF).
  \end{align*}
  Similarly, we have an isomorphism
  \[
  \mH^i(\tilde{s}_{\tau}^{-1}(V) \cap (X \times D), p_{\bA^1_{>0}}^* \cF) \cong \mH^i(p_{\bA^1_{>0}}(\tilde{s}_{\tau}^{-1}(V) \cap (X \times D)),  \cF).
  \]
  Therefore we have isomorphisms
  \[
  \mH^i (\cL_p) \cong \mH^i (\cR_p) \cong \cF_p
  \]
  which implies the claim.
\end{proof}

\begin{lem}\label{lem:spv}
  There exists an isomorphism
  \begin{equation}\label{eq:spv}
  (0_{E/E \times \bA^1}^* \tilde{s}_{\tau, p, *}  p_{\bA^1_{>0}}^* \cF )  |_{E_Z} \cong \iota_{C_{Z/X}, *}  \Sp_{Z/X}(\cF).
\end{equation}
\end{lem}

\begin{proof}
Consider the following commutative diagram:
\begin{equation}
\begin{split}
\xymatrix@C=35pt{
{}
& C_{Z/X} \pbcorner \ar@{^{(}->}[r]  \ar@{_{(}->}[d] \ar@{_{(}->}[ld]_-{\iota_{C_{Z/X}}} 
& M_{Z/X}^o \ar@{_{(}->}[d] 
& X \times \bA^1_{>0} \ar@{_{(}->}[l] \ar@{_{(}->}[ld]^-{\tilde{s}_{\tau,p}}  \\
E_Z \times \{ 0 \} \ar@{^{(}->}[r]
&E \times \{ 0 \} \ar@{^{(}->}[r]_-{0_{E/E \times \bA^1}} 
& E \times \bA^1 
&
}
\end{split}
\end{equation}
where the middle right vertical map $M_{Z/X}^o \to E \times \bA^1$ is the natural inclusion induced by the section $s$ (see \cite[Remark 5.1.1]{Ful84}).
By the definition of the specialization functor and the nearby cycle functor, we have an isomorphism
\[
\Sp_{Z/X}(\cF) \cong (C_{Z/X} \hookrightarrow M_{Z/X}^o)^* (X \times \bA^1_{>0} \hookrightarrow M_{Z/X}^o)_* p_{\bA^1_{>0}}^* \cF.
\]
Therefore the assertion follows from the base change theorem.
\end{proof}

\begin{proof}[Proof of Theorem \ref{thm:vansp}]
  We have already seen that the left-hand side of the statement is isomorphic to the complex \eqref{eq:inter1},
  which is isomorphic to the right-hand side of the statement by combining the isomorphism \eqref{eq:nattau}, Lemma \ref{lem:isoun} and Lemma \ref{lem:spv}.
\end{proof}

By the proof, the isomorphism \eqref{eq:vansp} in Theorem \ref{thm:vansp} is obtained by the composition
\begin{align}\label{eq:compvansp}
\begin{split}
  \FS_{E^\vee_Z}(\vphi_{\bar{s}}(\pi_{E^\vee}^!\cF) |_{E_Z^\vee})
  & \xrightarrow[\eqref{eq:revan}]{\sim} \FS_{E^\vee_Z}( (\hat{s}^! i_{X \times \ell_{\leq 0}, *} p_{\ell_ {\leq 0}}^! \cF) |_{E_Z^\vee}) \\
  & \xrightarrow[\eqref{eq:FSbc^*}]{\sim} \FS_{E^\vee}( \hat{s}^! i_{X \times \ell_{\leq 0}, *}  p_{\ell_ {\leq 0}}^! \cF) |_{E_Z} \\
  & \xrightarrow[\eqref{eq:FSbcc3}]{\sim} (\tilde{s}_* \FS_{X \times \bA^1}(  i_{X \times \ell_{\leq 0}, *}  p_{\ell_ {\leq 0}}^! \cF)) |_{E_Z} \\
  & \xrightarrow[\eqref{eq:FShalf}]{\sim} (\tilde{s}_* j_{X \times \bA^1_{>0}, *}  p_{\bA^1_{ >0}}^* \cF)|_{E_Z} \\
  & \xrightarrow[\eqref{eq:nattau}]{\sim} (\tilde{s}_{\tau, *}p_{\bA^1_{>0}}^*\cF) |_{E_Z} \\
  & \xrightarrow[\eqref{eq:inter3}]{\sim} (0_{E/E \times \bA^1}^* \tilde{s}_{\tau, p, *}  p_{\bA^1_{>0}}^* \cF)|_{E_Z} \\
  & \xrightarrow[\eqref{eq:spv}]{\sim} \iota_{C_{Z/X}, *} \Sp_{Z/X}(\cF).
\end{split}
\end{align}

\subsection{Dimensional reduction via Fourier--Sato transform}

We use the notation as in the previous subsection.
Applying the functor $\pi_{E_Z^\vee, !}$ to the natural map
$\varphi_{\bar{s}}(\pi_{E^\vee}^! \cF) |_{E_Z^\vee} \to (\pi_{E^\vee}^! \cF)|_{E_Z^\vee}$ defined in \eqref{eq:v*},
we obtain a map
\begin{equation}\label{eq:dimred!}
  \pi_{E_Z^\vee, !}(\varphi_{\bar{s}}(\pi_{E^\vee}^! \cF) |_{E_Z^\vee}) \to \cF |_Z
\end{equation}
which is proven to be isomorphic by Davison \cite[Theorem A.1]{Dav17}.
Similarly, applying the functor $\pi_{E_Z^\vee, *}$ to the natural map
$(\pi_{E^\vee}^! \cF)|_{E_Z^\vee}^! \to \varphi_{\bar{s}}(\pi_{E^\vee}^! \cF) |_{E_Z^\vee}^! $ defined in \eqref{eq:!v}, we obtain a map
\begin{equation}\label{eq:dimred*}
  \cF |^!_Z \to \pi_{E_Z^\vee, *}(\varphi_{\bar{s}}(\pi_{E^\vee}^! \cF) |_{E_Z^\vee}^!)[-2 \rank E]
\end{equation}
which is also isomorphic.
We claim that these isomorphisms are compatible with the isomorphism \eqref{eq:vansp}.
More precisely, the following statement holds:

\begin{prop}\label{prop:dimredcomp}
\hspace{2em}
  \begin{itemize}
    \item[(i)] The following diagram in $D_c^b(Z)$ commutes:
    \begin{equation}\label{eq:dimredcomp1}
    \begin{split}
    \xymatrix@C=40pt{
  \pi_{E_Z^\vee, !}(\varphi_{\bar{s}}(\pi_{E^\vee}^! \cF) |_{E_Z^\vee}) \ar[d]^-{\eqref{eq:FSbcc1}}_-{\simd} \ar[r]^-{\eqref{eq:dimred!}}_-{\sim}  & \cF|_Z \\
  0_{E_Z}^* \FS_{E_Z^\vee}(\varphi_{\bar{s}}(\pi_{E^\vee}^! \cF)|_{E_Z^\vee}) \ar[r]^-{(-1) \cdot \eqref{eq:vansp}}_-{\sim}
  & 0_{E_Z}^* \iota_{C_{Z/X}, *} \Sp_{Z/X}(\cF) \ar[u]^-{\eqref{eq:spres*}}_-{\simd}
    }
    \end{split}
  \end{equation}
  where $0_{E_Z} \colon Z \hookrightarrow E_Z$ is the zero section.
  \item[(ii)] The following diagram in $D_c^b(Z)$ commutes:
  \begin{equation}\label{eq:dimredcomp2}
  \begin{split}
    \xymatrix@C=40pt{
     \pi_{E_Z^\vee, *}(\varphi_{\bar{s}}(\pi_{E^\vee}^! \cF) |_{E_Z^\vee}^!)[-2 \rank E] \ar[d]^-{\eqref{eq:FSbcc2}}_-{\simd}
     & \cF |^!_Z \ar[d]_-{\eqref{eq:spres!}}^-{\simd} \ar[l]_-{\eqref{eq:dimred*}}^-{\sim} \\
    0_{E_Z}^! \FS_{E_Z^\vee}(\varphi_{\bar{s}}(\pi_{E^\vee}^! \cF)|_{E_Z^\vee}^!) \ar[r]^-{(-1)^{\rank E + 1} \cdot \eqref{eq:vansp}}_-{\sim}
    & 0_{E_Z}^!\iota_{C_{Z/X}, *} \Sp_{Z/X}(\cF).
    }
    \end{split}
  \end{equation}

  \end{itemize}
\end{prop}

The proof is technical, so we defer it to \S \ref{ssec:dimredcomp}.

\begin{rem}
  The proof of the above proposition does not depend on Davison's dimensional reduction theorem \cite[Theorem A.1]{Dav17}. Therefore we obtain a new proof of this theorem.
\end{rem}

  Now assume $X$ is smooth and consider the following composition
  \[
  \bQ_Z \cong \pi_{E_Z^\vee, !}(\varphi_{\bar{s}}(\pi_{E^\vee}^! \bQ_{X})|_{E_Z^\vee})
  \to \pi_{E_Z^\vee, *}(\varphi_{\bar{s}}(\pi_{E^\vee}^! \bQ_{X})|_{E_Z^\vee}^!) \cong \omega_{Z}[2 \rank E - 2 \dim X]
  \]
  where the first isomorphism is the inverse of the map \eqref{eq:dimred!}, the second map exists since the support of $\varphi_{\bar{s}}(\pi_{E^\vee}^! \bQ_{X})$ is contained in $E_Z^\vee$, and the last isomorphism is the inverse of the map \eqref{eq:dimred*}.
  We let
  \[
  e_{\varphi}(E, s) \in \HBM_{2 \dim X - 2 \rank E}(Z)
  \]
 denote the element corresponding to the above composition.

  \begin{cor}
    We have an equality
    \[
    e_{\varphi}(E, s) = (-1)^{\rank E} e_{\mathrm{loc}}(E, s)
    \]
    where $e_{\mathrm{loc}}(E, s)$ is the localized Euler class.
  \end{cor}

  \begin{proof}
    Consider the following composition:
    \[
    \bQ_Z \cong 0_{E_Z}^*\iota_{C_{Z/X}, *} \Sp_{Z/X}(\bQ_X)  \to 0_{E_Z}^!\iota_{C_{Z/X}, *} \Sp_{Z/X}(\bQ_X)[2 \rank E] \cong \omega_{Z}[2 \rank E - 2 \dim X]
    \]
    where the first isomorphism is the inverse of \eqref{eq:spres*}, the second morphism is \eqref{eq:smpurity}, and the last isomorphism is the inverse of \eqref{eq:spres!}.
    Proposition \ref{prop:dimredcomp} and Proposition \ref{prop:FSpure} imply that this composition corresponds to the element $(-1)^{\rank E} e_{\varphi}(E, s)$. On the other hand, Theorem \ref{thm:spcomm} implies that the above composition corresponds to the element $e_{\mathrm{loc}}(E, s)$.
  \end{proof}

We will prove the global version of the above corollary in the next section.


\section{Virtual fundamental classes via vanishing cycles}\label{sec:main}

\subsection{Virtual fundamental classes for quasi-smooth derived schemes}

We recall the construction of the virtual fundamental class for quasi-smooth derived schemes following \cite{BF97, LT98}.
Let $\bs{Y}$ be a separated quasi-smooth derived scheme over $\Spec \bC$ and we let $Y = t_0(\bs{Y})$ denote the classical truncation.
It is shown in \cite[Theorem 4.1]{BBJ19} that for each $p \in Y$ we can find a smooth scheme $U$, a vector bundle $E$ on $U$ and a section $s \in \Gamma(U, E)$ such that there exists an open immersion
\[
\bs{Z}(s) \hookrightarrow \bs{Y}
\]
where $\bs{Z}(s)$ is the derived zero locus of $s$.
In \cite[Defintion 3.10]{BF97}, Behrend and Fantechi introduced a conical closed substack of pure dimension zero $\fC_Y \subset t_0(\bfT[1] \bs{Y})$ of the classical truncation of the shifted tangent stack $\bfT[1] \bs{Y} = \dSpec_{\bs{Y}}(\Sym(\bL_{\bs{Y}}[-1]))$ called the intrinsic normal cone characterized by the following property:
For each open immersion $\bs{Z}(s) \hookrightarrow \bs{Y}$ as above, the restriction $\fC_Y |_{t_0(\bs{Z}(s))}$ is isomorphic to the following quotient stack
\[
[C_{Z(s) / U} / T_U |_{Z(s)}] \subset [E |_{Z(s)} / T_U  |_{Z(s)}] \cong t_0(\bfT[1] \bs{Z}(s))
\]
where $Z(s)$ is the classical truncation of $\bs{Z}(s)$ and the action of $T_U  |_{Z(s)}$ on the normal cone $C_{Z(s) / U}$ and the vector bundle $E_{Z(s)}$ is defined by the differential of the map $s \colon U \to E$.

Now assume that $\bL_{\bs{Y}} |_Y$ is represented by a two-term complex of locally free sheaves
\[
E^{\bullet} = [E^{-1} \to E^0]
\]
concentrated in degree $[-1, 0]$. This assumption is satisfied when $Y$ is quasi-projective.
Write $E_i \coloneqq (E^{-i})^\vee$ for $i = 0, 1$.
For such a resolution $E^{\bullet}$, we define a conical closed subscheme of pure dimension $\rank E_0$
\[
C_{E^{\bullet}} \subset E_1
\]
by the pullback of $\fC_{{Y}}$ along the projection
\[
E_1 \to [E_1/E_0] \cong t_0(\bfT[1]\bs{Y}).
\]
We define a class $[\bs{Y}]^{\vir}_{E^\bullet} \in A_{\vdim {\bs{Y}}}(Y)$ by the image of the class $[C_{E^{\bullet}}] \in A_{\rank E_0}(E_1)$ under the isomorphism $A_{\rank E_0}(E_1) \cong A_{\vdim {\bs{Y}}}(Y)$.
It is shown in \cite[Proposition 5.3]{BF97} that the class $[\bs{Y}]^{\vir}_{E^\bullet}$ is independent of the chosen resolution $E^\bullet$,
and the class $[\bs{Y}]^{\vir} \coloneqq [\bs{Y}]^{\vir}_{E^\bullet}$ is called the virtual fundamental class of $\bs{Y}$.

\subsection{Vanishing cycle complex associated with a $-1$-shifted cotangent scheme}

In \cite[Theorem 6.9]{BBDJS15}, the authors define natural perverse sheaves associated with oriented d-critical schemes.
We do not recall the notion of the d-critical scheme and the construction of this perverse sheaf here.
Instead, we explain some properties of the perverse sheaf
\[
\varphi_{\bfT^*[-1] \bs{Y}} \in \Perv(t_0(\bfT^*[-1] \bs{Y}))
\]
 associated with the natural oriented d-critical structure on the classical truncation
$t_0(\bfT^*[-1] \bs{Y})$ of the $-1$-shifted cotangent scheme of a quasi-smooth derived scheme $\bs{Y}$.
As explained in the previous subsection, $\bs{Y}$ is locally isomorphic to a derived zero locus $\bs{Z}(s)$ where $s$ is a section of a vector bundle $E$ on a smooth scheme $U$.
Note that we have a natural isomorphism
\[
t_0(\bfT^*[-1]\bs{Z}(s)) \cong \Crit(\bar{s})
\] where $\bar{s} \colon E^\vee \to \bA^1$ is the regular function corresponding to the section $s$.
The first property of $\varphi_{\bfT^*[-1] \bs{Y}}$ that we use in this paper is that for each open inclusion $\bs{\iota} \colon \bs{Z}(s) \hookrightarrow \bs{Y}$ as above such that $E$ and $\Omega_U$ are trivial vector bundles, we have a natural isomorphism
\begin{equation}\label{eq:isomoritriv}
\eta_{\bs{\iota}} \colon \varphi_{\bfT^*[-1] \bs{Y}}  |_{t_0(\bfT^*[-1]\bs{Z}(s))} \cong \varphi_{\bar{s}}(\bQ_{E^\vee}[\dim U + \rank E]).
\end{equation}
See \cite[Lemma 2.19]{Kin21} for the proof.

Now write $\widetilde{Y} = t_0(\bfT^*[-1] \bs{Y})$ and $Y = t_0(\bs{Y})$.
We let $\pi \colon \widetilde{Y} \to Y$ denote the natural projection.
The second property of $\varphi_{\bfT^*[-1] \bs{Y}}$ that we use is the following.

\begin{thm}[{\cite[Theorem 3.1]{Kin21}}]\label{thm:dimred}
  There exist natural isomorphisms which we call the dimensional reduction isomorphisms
\[
\gamma \colon \pi_! \varphi_{\bfT^*[-1] \bs{Y}} \cong \bQ_Y[\vdim \bs{Y}], \ \bar{\gamma} \colon \pi_* \varphi_{\bfT^*[-1] \bs{Y}} \cong \omega_Y[- \vdim \bs{Y}]
\]
such that for each open inclusion $\bs{\iota} \colon \bs{Z}(s) \hookrightarrow \bs{Y}$, the following diagrams commute:
\begin{equation}\label{eq:dimredlocmod1}
\xymatrix@C=50pt{
 (\pi_! \varphi_{\bfT^*[-1] \bs{Y}}) |_{t_0(\bs{Z}(s))} \ar[d]_-{(\pi |_{t_0(\bfT^*[-1]\bs{Z}(s))})_!(\eta_{\bs{\iota}})} \ar[r]^-{\gamma |_{t_0(\bs{Z}(s))}} & \bQ_{t_0(\bs{Z}(s))}[\vdim \bs{Y}] \ar@{=}[d] \\
 (\pi |_{t_0(\bfT^*[-1]\bs{Z}(s))})_! \varphi_{\bar{s}}(\bQ_{E^\vee}[\dim U + \rank E]) \ar[r]^-{\eqref{eq:dimred!}} & \bQ_{t_0(\bs{Z}(s))}[\vdim \bs{Y}]
}
\end{equation}

\begin{equation}\label{eq:dimredlocmod2}
\xymatrix@C=65pt{
 (\pi_* \varphi_{\bfT^*[-1] \bs{Y}}) |_{t_0(\bs{Z}(s))} \ar[d]_-{(\pi |_{t_0(\bfT^*[-1]\bs{Z}(s))})_*(\eta_{\bs{\iota}})} \ar[r]^-{(-1)^{\dim E \cdot (\dim E -1)/2} \bar{\gamma} |_{t_0(\bs{Z}(s))}}
 & \omega_{t_0(\bs{Z}(s))}[- \vdim \bs{Y}] \ar@{=}[d] \\
 (\pi |_{t_0(\bfT^*[-1]\bs{Z}(s))})_* \varphi_{\bar{s}}(\bQ_{E^\vee}[\dim U + \rank E])
 & \omega_{t_0(\bs{Z}(s))}[- \vdim \bs{Y}] \ar[l]_-{\eqref{eq:dimred*}} .
}
\end{equation}

\end{thm}

The isomorphism $\gamma$ is constructed in \cite[Theorem 3.1]{Kin21} and the commutativity of the diagram \eqref{eq:dimredlocmod1} follows from its proof.
The isomorphism $\bar{\gamma}$ can be constructed using the Verdier self-duality of $\varphi_{\bfT^*[-1]\bs{Y}}$, and the commutativity of the diagram \eqref{eq:dimredlocmod2} follows from the Verdier dual of the diagram \eqref{eq:dimredlocmod1}, the commutativity of the diagram \eqref{eq:vannatdual},
 and the discussion after \cite[Theorem 2.17]{Kin21}.

 Now consider the following composition
 \begin{equation}\label{eq:pervVFC}
   \bQ_Y[\vdim \bs{Y}] \xleftarrow[\cong]{\gamma} \pi_! \varphi_{\bfT^*[-1] \bs{Y}} \to \pi_* \varphi_{\bfT^*[-1] \bs{Y}} \xrightarrow[\cong]{\bar{\gamma}} \omega_Y[- \vdim \bs{Y}].
 \end{equation}
We let $e_{\varphi}(\bfT^*[-1] \bs{Y}) \in \HBM_{2\vdim \bs{Y}}(Y)$ denote the element corresponding to this composition.

\begin{thm}\label{thm:main}
  Assume that $\bL_{\bs{Y}} |_Y$ is represented by a two-term complex of locally free sheaves.
  Then the following equality holds:
  \begin{equation}
    e_{\varphi}(\bfT^*[-1] \bs{Y}) = (-1)^{\vdim {\bs{Y}} \cdot (\vdim \bs{Y} -1)/2} \cl_Y([\bs{Y}]^{\vir}).
  \end{equation}
\end{thm}

In other words, the above theorem gives a new construction of the virtual fundamental class in the Borel--Moore homology under a mild assumption that is satisfied when $Y$ is quasi-projective.

\subsection{Proof of Theorem \ref{thm:main}}\label{ssec:main}

Take a global resolution $\bL_{\bs{Y}} |_Y \cong E^{\bullet} = [E^{-1} \to E^0]$ and let
$\iota_{E^{\bullet}} \colon t_0(\bfT^*[-1]\bs{Y}) \hookrightarrow E^{-1}$ be the natural inclusion.
The following lemma is the key ingredient in the proof of Theorem \ref{thm:main}.
\begin{lem}\label{lem:keying}
The support of the complex $\FS_{E^{-1}}(\iota_{E^{\bullet}, *} \varphi_{\bfT^*[-1] \bs{Y}})$ is contained in the cone $C_{E^{\bullet}}$.
\end{lem}

\begin{proof}
  Take another global resolution $\bL_{\bs{Y}}|_Y \cong F^{\bullet}$ by a two-term complex of locally free sheaves.
  We use the same notation as above for $F^{\bullet}$.
  We claim that the statements for $E^{\bullet}$ and $F^{\bullet}$ are equivalent.
  Arguing as the proof of \cite[Proposition 5.3]{BF97}, we may assume that there exists a strict monomorphism of complexes $F^{\bullet} \hookrightarrow E^{\bullet}$.
  Let $q \colon E_{1} \to F_1$ be the dual morphism to the map $F^{-1} \hookrightarrow E^{-1}$ which is smooth surjective by assumption.
  Then we have an equality $C_{E^{\bullet}} = q^{-1}(C_{F^{\bullet}})$. Thus the desired equivalence follows from the isomorphism \eqref{eq:FSbcc1}.

Now we return back to the proof of the lemma.
Since the statement can be checked locally, we may assume $\bs{Y} = \bs{Z}(s)$ for some section $s \in \Gamma(U, E)$ of a vector bundle $E$ on a smooth scheme $U$.
Further, we may take the canonical resolution $[E^\vee |_{t_0(\bs{Z}(s))} \to \Omega_U|_{t_0(\bs{Z}(s))}] \simeq \bL_{\bs{Z}(s)} |_{t_0(\bs{Z}(s))}$ as the global resolution $E^{\bullet}$. 
Then the claim follows from the isomorphism \eqref{eq:isomoritriv} and Theorem \ref{thm:vansp}.
\end{proof}

Let $0_{C_{E^{\bullet}}} \colon Y \hookrightarrow C_{E^{\bullet}}$ and $0_{E_1} \colon Y \hookrightarrow E_1$ be the zero sections and $i_{{C_{E^{\bullet}}}} \colon C_{E^{\bullet}} \hookrightarrow E_1$ be the natural inclusion.
Consider the following diagram in $D^b_c(Y)$
\begin{equation}\label{eq:coneclassconst}
\xymatrix{
0_{C_{E^{\bullet}}}^*(\FS_{E^{-1}}(\iota_{E^{\bullet}, *} \varphi_{\bfT^*[-1] \bs{Y}})) \otimes 0_{C_{E^{\bullet}}}^! \bQ_{C_{E^{\bullet}}} \ar[r]^-{\eqref{eq:purity}}
& 0_{C_{E^{\bullet}}}^!(\FS_{E^{-1}}(\iota_{E^{\bullet}, *} \varphi_{\bfT^*[-1] \bs{Y}}))  \\
\pi_!\varphi_{\bfT^*[-1] \bs{Y}} \otimes 0_{C_{E^{\bullet}}}^! \bQ_{C_{E^{\bullet}}} \ar[u]_-{\eqref{eq:FSbcc1}}^-{\simd} \ar[d]^-{\gamma \otimes \id}_-{\simd}
& \pi_*\varphi_{\bfT^*[-1] \bs{Y}}[-2 \rank E_1] \ar[u]_-{\eqref{eq:FSbcc2}}^-{\simd} \ar[d]^-{\bar{\gamma} }_-{\simd} \\
0_{C_{E^{\bullet}}}^! \bQ_{C_{E^{\bullet}}}[\vdim \bs{Y}] \ar@{-->}[r]^-{\alpha}
& \omega_{Y}[-\vdim \bs{Y} - 2 \rank E_1]
}
\end{equation}
where the bottom horizontal arrow $\alpha$ is defined so that the diagram becomes commutative.
Let $\pi_{C_{E^{\bullet}}} \colon C_{E^{\bullet}} \to Y$ be the natural projection.
We have a natural isomorphism ${\pi_{C_{E^{\bullet}}}}_! \omega_{C_{E^{\bullet}}} \cong 0_{C_{E^{\bullet}}}^! \omega_{C_{E^{\bullet}}} $ constructed in \eqref{eq:homotopyinv2}.
Therefore the map $\alpha$ induces a map
\[
\bQ_{C_{E^{\bullet}}}[\vdim \bs{Y}] \to \omega_{C_{E^{\bullet}}}[-\vdim \bs{Y} - 2 \rank E_1]
\]
or equivalently an element in $ \HBM_{2 \rank E_0}(C_{E^{\bullet}}) \cong \HBM_0(\fC_Y) $ which is denoted by $[\fC_Y]^{E^{\bullet}}_{ \varphi}$.

\begin{lem}
  The class $[\fC_Y]^{E^{\bullet}}_{ \varphi}$ is independent of the choice of the resolution $E^{\bullet}$.
\end{lem}

\begin{proof}
Arguing as the proof of Lemma \ref{lem:keying}, we need to prove the equality 
$[\fC_Y]^{E^{\bullet}}_{ \varphi} = [\fC_Y]^{F^{\bullet}}_{ \varphi}$
for resolutions $E^{\bullet}$ and $F^{\bullet}$ of the complex $\bL_{\bs{Z}(s)} |_{Z(s)}$ such that there exists a strict monomorphism $F^{\bullet} \hookrightarrow E^{\bullet}$.
We keep the notation from the proof of Lemma \ref{lem:keying} and
  let $q_C \colon C_{E^{\bullet}} \to C_{F^{\bullet}}$
  be the restriction of $q$.
   Consider the following diagram in $D^b_c(Y)$:
  \[
  \begin{tikzcd}[scale cd=0.7]
  0_{C_{E^{\bullet}}}^*(\FS_{E^{-1}}(\iota_{E^{\bullet}, *} \varphi_{\bfT^*[-1] \bs{Y}})) \otimes 0_{C_{E^{\bullet}}}^! \bQ_{C_{E^{\bullet}}} \arrow[rrr] \arrow[ddd, shift right =10ex, "\simd"] \arrow[rd, "\sim", "\text{(A)}"']
  &[-130pt]
  &[-10pt]
  &[-85pt] 0_{C_{E^{\bullet}}}^!(\FS_{E^{-1}}(\iota_{E^{\bullet}, *} \varphi_{\bfT^*[-1] \bs{Y}})) \arrow[ld, "\sim"',"\text{(B)}"] \arrow[ddd, shift left =5ex, "\simd"] \\
  &0_{C_{F^{\bullet}}}^*(\FS_{F^{-1}}({\iota_{F^{\bullet}, *}} \varphi_{\bfT^*[-1] \bs{Y}})) \otimes 0_{C_{F^{\bullet}}}^! \bQ_{C_{F^{\bullet}}}[-2 \dim q] \arrow[r] \arrow[d, "\simd"]
  & 0_{C_{F^{\bullet}}}^!(\FS_{F^{-1}}({\iota_{F^{\bullet}, *}} \varphi_{\bfT^*[-1] \bs{Y}}))[-2 \dim q] \arrow[d, "\simd"]
  & \\
  & 0_{C_{F^{\bullet}}}^! \bQ_{C_{F^{\bullet}}}[\vdim \bs{Y} - 2 \dim q] \arrow[ld,"\sim"',"\text{(C)}", start anchor={[xshift=0pt]},
  end anchor={[xshift=-30pt]}] \arrow[r, dashed]
  & \omega_{Y}[-\vdim \bs{Y} - 2 \rank E_1] \arrow[rd, equal]
  & \\
0_{C_{E^{\bullet}}}^! \bQ_{C_{E^{\bullet}}}[\vdim \bs{Y}] \quad \quad \quad \quad \quad \quad \quad \quad \ar[rrr,dashed, start anchor={[xshift=-55pt]},
end anchor={[xshift=0pt]}]
  &
  &
  & \omega_{Y}[-\vdim \bs{Y} - 2 \rank E_1].
\end{tikzcd}
  \]
  where the inner and outer rectangles are the diagram \eqref{eq:coneclassconst} for $F^{\bullet}$ and $E^{\bullet}$ respectively, the map (A) is defined using the isomorphism \eqref{eq:FSbcc1} and the natural isomorphism $q_C^! \bQ_{C_{F^{\bullet}}}[-2 \dim q] \cong \bQ_{C_{E^{\bullet}}}$,
  the map (B) is defined using the isomorphism \eqref{eq:FSbcc2}, and the map (C) is defined by the natural isomorphism  $q_C^! \bQ_{C_{F^{\bullet}}}[-2 \dim q] \cong \bQ_{C_{E^{\bullet}}}$ again.
  The commutativity of the left (resp. right) trapezoid follows from the commutative diagram \eqref{eq:FSbcc1ass} (resp. the diagram \eqref{eq:FSbcc2ass}).  The commutativity of the upper trapezoid follows from Proposition \ref{prop:FSpure}.
  Therefore we obtain the commutativity of the lower trapezoid, which implies the lemma.
\end{proof}

\begin{prop}\label{prop:coneclass}
  The class $[\fC_Y]^{E^{\bullet}}_{ \varphi}$ is equal to $(-1)^{\vdim \bs{Y} \cdot (\vdim \bs{Y} -1)/2}[\fC_Y]$ where $[\fC_Y]$ is the fundamental class of $\fC_Y$.
\end{prop}

\begin{proof}
  Since the statement can be checked locally on $\bs{Y}$, we may assume $\bs{Y} = \bs{Z}(s)$ where $s$ is a section of a vector bundle $E$ on a smooth scheme $U$. We write $Z(s) \coloneqq t_0(\bs{Z}(s))$.
  Using the previous lemma, we may take the canonical resolution $[E^\vee |_{Z(s)} \to \Omega_U|_{Z(s)}] \simeq \bL_{\bs{Z}(s)} |_{Z(s)}$ as the global resolution $E^{\bullet}$.
  In this case, the cone $C_{E^\bullet}$ is identified with the normal cone $C_{Z(s) / U} \subset E$ and the embedding
  $\iota_{E^{\bullet}} \colon t_0(\bfT^*[-1]\bs{Z}(s)) \hookrightarrow E^\vee$ is identified with the natural inclusion $\Crit(\bar{s}) \hookrightarrow E^\vee$.
  Therefore the isomorphism \eqref{eq:isomoritriv} and Theorem \ref{thm:vansp} implies an isomorphism
  \[
  \FS_{E^\vee}{ (\iota_{E^{\bullet}, *}} \varphi_{\bfT^*[-1] \bs{Y}})) \cong \Sp_{Z(s) / U}(\bQ_U[\vdim \bs{Z}(s)]).
  \]
  Further, Proposition \ref{prop:dimredcomp} and the commutativity of the diagrams \eqref{eq:dimredlocmod1} and \eqref{eq:dimredlocmod2} imply that the outer square of the diagram \eqref{eq:coneclassconst} is identified with the following diagram up to the shift by $\vdim \bs{Z}(s)$:
  \begin{equation}
   \def\objectstyle{\scriptstyle}
  \def\labelstyle{\scriptscriptstyle}
    \xymatrix@C=70pt{
    0_{C_{Z(s)/U}}^*(\Sp_{Z(s) / U}(\bQ_U)) \otimes 0_{C_{Z(s)/U}}^! \bQ_{C_{Z(s)/U}} \ar[r]^(.57){(-1)^{\dim E \cdot (\dim E -1)/2}\eqref{eq:purity}}   \ar[d]^-{(-1) \cdot\eqref{eq:spres*}} _-{\simd}
    & 0_{C_{Z(s)/U}}^! (\Sp_{Z(s) / U}(\bQ_U)) \ar[d]^-{(-1)^{\rank E + 1} \cdot \eqref{eq:spres!}} _-{\simd} \\
     0_{C_{Z(s)/U}}^! \bQ_{C_{Z(s)/U}} \ar@{-->}[r]^-{\alpha[- \vdim \bs{Z}(s)]}
    & \omega_{Z(s)}[-2 \rank E]
    }
  \end{equation}
  where $0_{C_{Z(s)/U}} \colon Z(s) \hookrightarrow C_{Z(s)/U}$ is the zero section.
  Therefore Theorem \ref{thm:spcomm} and the construction of the map $\Sp_{Z(s)/U}^{\BM}$ in \S \ref{ssec:spmap} implies an equality
  \begin{align*}
  (C_{Z(s)/U} \to \fC_Y)^! [\fC_Y]^{E^{\bullet}}_{ \varphi} &= (-1)^{\vdim \bs{Y} \cdot (\vdim \bs{Y} -1)/2}\Sp_{Z(s)/U}^{\Chow}([U]) \\
   &= (-1)^{\vdim \bs{Y} \cdot (\vdim \bs{Y} -1)/2} [C_{Z(s)/U}]
\end{align*}
  which implies the proposition.
\end{proof}

\begin{proof}[Proof of Theorem \ref{thm:main}]
Consider the following diagram:
\[
  \def\objectstyle{\scriptstyle}
\def\labelstyle{\scriptscriptstyle}
\xymatrix@C=30pt{
0_{C_{E^{\bullet}}}^*(\FS_{E^{-1}}(\iota_{E^{\bullet}, *} \varphi_{\bfT^*[-1] \bs{Y}})) \otimes 0_{C_{E^{\bullet}}}^! \bQ_{C_{E^{\bullet}}} \ar[r]^-{\eqref{eq:purity}} \ar@/_110pt/[dddd]
& 0_{C_{E^{\bullet}}}^! \FS_{E^{-1}}(\iota_{E^{\bullet}, *} \varphi_{\bfT^*[-1] \bs{Y}})\ar@/^90pt/[dddd]  \\
0_{E_1}^*(\FS_{E^{-1}}(\iota_{E^{\bullet}, *} \varphi_{\bfT^*[-1] \bs{Y}})) \otimes 0_{E_1}^! \bQ_{E_1} \ar[r]^-{\eqref{eq:purity}} \ar[u]
& 0_{E_1}^! \FS_{E^{-1}}(\iota_{E^{\bullet}, *} \varphi_{\bfT^*[-1] \bs{Y}}) \ar[u]^-{\simd} \\
\pi_!\varphi_{\bfT^*[-1] \bs{Y}} \otimes 0_{E_1}^! \bQ_{E_1} \ar[u]_-{\eqref{eq:FSbcc1}}^-{\simd} \ar[d]^-{\gamma \otimes \id}_-{\simd}
& \pi_*\varphi_{\bfT^*[-1] \bs{Y}}[-2 \rank E_1] \ar[u]_-{\eqref{eq:FSbcc2}}^-{\simd} \ar[d]^-{\bar{\gamma}}_-{\simd} \\
0_{E_1}^! \bQ_{E_1}[\vdim \bs{Y}] \ar[r]^-{e_{\varphi}(\bfT^*[-1]\bs{X})[-2 \rank E_1]} \ar[d]
& \omega_{Y}[-\vdim \bs{Y} - 2 \rank E_1] \ar@{=}[d] \\
 0_{C_{E^{\bullet}}}^! \bQ_{C_{E^{\bullet}}}[\vdim \bs{Y}] \ar[r]^-{\alpha}
& \omega_{Y}[-\vdim \bs{Y} - 2 \rank E_1]
}
\]
where the outer square is equal to the outer square of \eqref{eq:coneclassconst}.
The commutativity of the upper rectangle and the left and right subdiagrams is obvious, and the commutativity of the middle rectangle follows from Proposition \ref{prop:FSpure}. This implies the commutativity of the bottom rectangle hence we conclude that the theorem holds using Proposition \ref{prop:coneclass}.
\end{proof}


\section{DT4 virtual classes and categorifications}\label{sec:future}

In this section, we discuss some possible generalizations of the sheaf theoretic construction of the virtual fundamental class we have seen in Theorem \ref{thm:main}.
In \S \ref{ssec:DT4} we formulate a conjecture closely related to the dimensional reduction theorem \cite[Theorem 3.1]{Kin21} and propose a sheaf theoretic construction of the DT4 invariant assuming this conjecture. We will see that this construction recovers the square root Euler class up to sign in a special case, which gives evidence of this conjecture.
In \S \ref{ssec:categorify}, we will discuss a categorification of Theorem \ref{thm:main} and its relation with the categorical Donaldson--Thomas invariant introduced by Toda \cite{Tod19}.

\subsection{DT4 invariants}\label{ssec:DT4}

Let $\bs{Y}$ be a $-2$-shifted symplectic derived scheme over $\Spec \bC$ (e.g. the fine moduli space of coherent sheaves on a complex Calabi--Yau fourfold). We refer the reader to \cite{PTVV13} for the definition of the shifted symplectic structure.
The $-2$-shifted symplectic structure defines a symmetry of the cotangent complex $\bL_{\bs{Y}}^{\vee} \cong \bL_{\bs{Y}}[-2]$.
Further assume that $\bs{Y}$ is equipped with a choice of orientation.  
Based on a previous work by Cao--Leung \cite{CL14} and Borisov--Joyce \cite{BJ17}, Oh and Thomas in \cite{OT20} defined a deformation-invariant class
\[
[\bs{Y}]_{\vir}^{\DT4} \in A_{\vdim{\bs{Y}}/2}(Y)[\tfrac{1}{2}]
\]
called the DT4 virtual class under the assumption that $Y = t_0(\bs{Y})$ is quasi-projective.
This virtual class can be used to define Donaldson--Thomas type invariants for Calabi--Yau fourfolds.

We now propose a sheaf theoretic approach to the DT4 virtual class.
We define a quadratic function $q$ on the $-1$-shifted cotangent stack $\widetilde{Y} = t_0(\bfT^*[-1]\bs{Y})$ by the following composition
\[
q \colon \widetilde{Y} \to t_0(\Tot_{\bs{Y}}(\bL_{\bs{Y}}[-1] \oplus \bL_{\bs{Y}}[-1])) \to \bA^1
\]
where the first map is the diagonal embedding over $Y$ and the latter map is defined using the $-2$-shifted symplectic form on $\bs{Y}$.
Let $\varphi_{\bfT^*[-1]\bs{Y}}$ be the perverse sheaf introduced in \cite[Theorem 4.8]{BBBBJ15} associated with the canonical oriented $-1$-shifted symplectic structure on $\bfT^*[-1]\bs{Y}$.
We denote by $\pi \colon \widetilde{Y} \to Y$ the natural projection and $0_{Y} \colon Y \to \widetilde{Y}$ the zero section.

\begin{conj}\label{conj:dimred-2}
\hspace{2em}
  \begin{itemize}
    \item[(i)] There exist natural isomorphisms
    \[
\gamma \colon 0_{Y}^! \varphi_{\bfT^*[-1] \bs{Y}} \cong \bQ_Y[\vdim \bs{Y}], \ \bar{\gamma} \colon 0_Y^* \varphi_{\bfT^*[-1] \bs{Y}} \cong \omega_Y[- \vdim \bs{Y}].
    \]
    \item[(ii)] There exists a natural isomorphism
    \[
    \delta \colon 0_{Y}^! \varphi_q(\varphi_{\bfT^*[-1] \bs{Y}}) \cong \omega_{Y}
    \]
    which depends on the chosen orientation.
  \end{itemize}
\end{conj}

\begin{rem}
   Note that the first statement does not follow from Theorem \ref{thm:dimred} since $\bs{Y}$ is not quasi-smooth now.
    We believe that this statement can be generalized to an arbitrary derived Artin $1$-stack whose cotangent complex is perfect of amplitude $[-2, 1]$ (or an arbitrary locally finitely presented derived Artin stack once the work \cite{BBBBJ15} has generalized to higher Artin stacks).
\end{rem}

\begin{rem}
    We do not have a natural map $0_Y^! \varphi_{\bfT^*[-1] \bs{Y}} \to 0_Y^*  \varphi_{\bfT^*[-1] \bs{Y}}$ in general since $0_Y$ is not an embedding when $\bs{Y}$ is not quasi-smooth.
\end{rem}

Roughly speaking, the DT4 virtual class can be regarded as a ``square root’’ of the virtual fundamental class though we do not know whether we can define virtual fundamental classes for non-quasi-smooth derived schemes.
Therefore we expect that there exists a some Verdier-self-dual complex that sits between $\bQ_Y[\vdim \bs{Y}]$ and $\omega_Y[-\vdim \bs{Y}]$ which recovers DT4 virtual class. Our proposal is that the complex is $\varphi_q(\varphi_{\bfT^*[-1] \bs{Y}})$ though it lives in $\widetilde{Y}$ rather than $Y$.

  Now assume that Conjecture \ref{conj:dimred-2} is true. Consider the following composition
  \begin{align*}
    \bQ_Y[\vdim \bs{Y}]
    \xrightarrow[\cong]{\gamma^{-1}} 0_{Y}^! \varphi_{\bfT^*[-1] \bs{Y}}
    \to 0_Y ^! \varphi_q (\varphi_{\bfT^*[-1] \bs{Y}})
    \xrightarrow[\cong]{\delta} \omega_Y
  \end{align*}
  where the second map is constructed using the natural transform
  $(q^{-1}(0) \hookrightarrow \widetilde{Y})^! \to \varphi_q$.
  Let
  $\sqrt{e_{\varphi}(\bfT^*[-1]\bs{Y})} \in \HBM_{\vdim{\bs{Y}}}(Y)$ be the element corresponding to the above composition.

  \begin{conj}\label{conj:DT4}
  Assume that $Y$ is quasi-projective. Then there exists a universal sign $\epsilon_{\vdim{\bs{Y}}} \in \{-1, 1\}$ that only depends on $\vdim{\bs{Y}}$ such that we have an equality
  \[
  \epsilon_{\vdim{\bs{Y}}} \cdot \sqrt{e_{\varphi}(\bfT^*[-1]\bs{Y})} = \cl_Y([\bs{Y}]_{\vir}^{\DT4})
  \]
  where $\cl_Y \colon A_{\vdim{\bs{Y}}/2}(Y)[\tfrac{1}{2}] \to \HBM_{\vdim \bs{Y}}(Y)$ is the cycle class map (recall that the Borel--Moore homology is taken in the rational coefficient in this paper).
\end{conj}

We prove this conjecture in a special case (up to sign).
Let $\bs{U}$ be an oriented $-2$-shifted symplectic derived scheme such that the classical truncation $U = t_0(\bs{U})$ is  smooth and that there exists a resolution of the cotangent complex of the following form
\[
\bL_{\bs{U}} |_U \cong [T_U \xrightarrow{0} E \xrightarrow{0} \Omega_U]
\]
where $E$ is a vector bundle on $U$.
The oriented $-2$-shifted symplectic structure on $\bs{U}$ induces a non-degenerate quadratic form $\tilde{q}$ on $E$ and a choice of orientation $o \colon \cO_U \cong \det(E)$.
In this case, the DT4 virtual class $[\bs{Y}]_{\vir}^{\DT4}$ is equal to $ \sqrt{e(E)} \cap [U]$ if $\rank E$ is even where
$\sqrt{e(E)}$ is the Edidin--Graham's square root Euler class (see \cite{EG95} or \cite[\S 3]{OT20}) and zero if $\rank E$ is odd.

\begin{prop}\label{prop:DT4}
  Let $\bs{U}$ be a $-2$-shifted symplectic derived scheme as above.
  \begin{itemize}
\item[(i)]  Conjecture \ref{conj:dimred-2} holds for $\bs{U}$.
\item[(ii)] Assume that $U$ is connected. Then there exists a choice of sign $\epsilon \in \{-1, 1\}$ such that the following equality holds:
\[
\epsilon \cdot \sqrt{e_{\varphi}(\bfT^*[-1]\bs{U})} = \cl_U([\bs{U}]_{\vir}^{\DT4}).
\]
\end{itemize}
\end{prop}

\begin{proof}[Sketch of the proof]
Firstly note that we have a natural isomorphism $t_0(\bfT^*[-1]\bs{U}) \cong [E/T_U]$ where $T_U$ acts on $E$ trivially.
The smoothness of $t_0(\bfT^*[-1]\bs{U})$ and \cite[Example 2.15]{Joy15} imply that the d-critical structure on $t_0(\bfT^*[-1]\bs{U})$ is trivial. Therefore there exists a local system $\cL$ on $t_0(\bfT^*[-1]\bs{U})$ such that we have an isomorphism
\[
\varphi_{\bfT^*[-1] \bs{U}} \cong \cL[ \rank E].
\]
Detailed analysis on the canonical orientation for $\bfT^*[-1]\bs{U}$ as in \cite{Kin21} shows that the local system $\cL$ is trivial.
Therefore the first statement follows immediately.

Now we prove the second statement. We only treat the case when $\rank E$ is even as the odd case can be proved in a similar way.
Let $0_E \colon U \hookrightarrow E$ be the zero section.
Then we can easily see that the following diagram commutes
\[
\xymatrix@C=70pt{
0_E^! \bQ_{E}[2 \rank E] \ar[r] \ar[d]_-{\simd}
& 0_E^! \varphi_{\tilde{q}} (\bQ_{E})[2 \rank E] \ar[d]_-{\simd}  \\
\bQ_U \ar[r]^-{ \cap \sqrt{e_{\varphi}(\bfT^*[-1]\bs{U})}}
& \omega_U[\rank E - 2 \dim U].
}
\]
where the top vertical arrow is induced by the natural transformation $(\tilde{q}^{-1}(0) \hookrightarrow E)^! \to \varphi_{\tilde{q}}$.
As proved in \cite[Proposition 5]{EG95}, there exists a smooth proper map $f \colon F \to U$ with the following property:
\begin{itemize}
  \item The quadratic bundle $f^*E$ admits a positive maximal isotropic subbundle $\Lambda \subset f^*E$.
  \item An equality $f^![\bs{U}]_{\vir}^{\DT4} = e(\Lambda) \cap [F] $ holds.
  \item The map $f^! \colon \HBM_*(U) \to \HBM_{* + 2 \dim f}(F)$ is injective.
\end{itemize}
Therefore the proposition follows from the following lemma.

\begin{lem}
  Let $U$ be a connected and separated scheme of finite type over $\Spec \bC$, $(E, q, o)$ be an oriented non-degenerate quadratic vector bundle on $U$ which admits a positive maximal isotropic subbundle $\Lambda \subset E$. Then there exists a choice of sign $\epsilon \in \{1, -1 \}$ such that the following diagram commutes:
  \[
  \xymatrix{
  0_E^! \bQ_{E}[2 \rank E] \ar[r] \ar[d]_-{\simd}
  & 0_E^! \varphi_{{q}} (\bQ_{E})[2 \rank E] \ar[d]_-{\simd}  \\
  \bQ_U \ar[r]^-{ \epsilon e(\Lambda)}
  & \bQ_U[\rank E].
  }
  \]
\end{lem}
\begin{proof}
Write $r = \rank E$. Consider the following diagram in $D^b(E, \bZ)$
\[
\xymatrix{
&
&
& \bZ_U[-2r] \ar[d]^-{\simd} \ar[llldd]_(.5){(-1)^{r/2}(U \hookrightarrow \Lambda)_* \ \ \ } \ar@{-->}[rrdd]^-{}
&
& \\
{}
&
&
& \varphi_q(\bZ_E |^!_{U}) \ar[ld] \ar[rd]
&
& \\
\bZ_{\Lambda}[-r] \ar[r]^-{\sim} \ar@/_24pt/[rrrrr]^-{(U \hookrightarrow \Lambda) ^* }
& \bZ_{E} |^!_{\Lambda} \ar[r]^-{\sim}
& \varphi_q(\bZ_{E} |^!_{\Lambda}) \ar[rr]
&
&\varphi_q(\bZ_{E} ) \ar[r]^-{\sim}
&\bZ_U[-r]
}
\]
where the right dashed arrow is defined so that the upper right subdiagram commutes.
It is clear that the left subdiagram and the middle triangle commute.
Therefore we need to prove the commutativity of the bottom subdiagram up to sign.
To do this, we need to show that the natural map
\begin{equation}\label{eq:harumaki}
\varphi_q(\bZ_{E} |^!_{\Lambda}) |_U \to \varphi_q(\bZ_{E} )|_{U}
\end{equation}
is an isomorphism.
Take a point $u \in U$ and let $j \colon E \setminus \Lambda \hookrightarrow E$ be the natural open inclusion.
Then the claim that the map \eqref{eq:harumaki} is isomorphic at $u$ is equivalent to the vanishing of the complex $\varphi_q(j_* \bZ_{E \setminus \Lambda})_u$ which is further equivalent to that the natural map
\begin{equation}\label{eq:harumaki2}
(j_* \bZ_{E \setminus \Lambda})_u \to \psi_q(j_* \bZ_{E \setminus \Lambda})_u \cong \psi_q(\bZ_{E})_u
\end{equation}
is isomorphic. Let $M_{q, u} \subset E$ be a Milnor fiber of $q$ at $E$. Then it is clear that the following inclusions
\[
E_u \setminus \Lambda_u \hookleftarrow M_{q, u} \cap (E_u \setminus \Lambda_u) \hookrightarrow M_{q, u}
\]
are homotopy equivalences. This implies that the map \eqref{eq:harumaki2} is isomorphic as the stalk of the nearby cycle complex is isomorphic to the cohomology of the Milnor fiber.
\end{proof}
This completes the proof of Proposition \ref{prop:DT4}.
\end{proof}

\subsection{Categorification}\label{ssec:categorify}

It is known that the structure sheaf of a quasi-smooth derived scheme is a categorification of the virtual fundamental class (see \cite{CFK09} and \cite{FG10}). Further, it is expected that there is a categorification of the perverse sheaf associated with $(-1)$-shifted symplectic derived Artin stacks as a dg-category (see \cite{Tod19}).
Therefore it is natural to expect that Theorem \ref{thm:main} has a categorification, which we briefly discuss now.
To simplify the discussion, we only consider the schematic case.

 Let $\bs{X}$ be a $-1$-shifted symplectic derived scheme equipped with a fixed orientation.
 It is conjectured in \cite{Tod19} that there is a natural dg-category
 \[
 \mathcal{DT}(\bs{X})
 \]
 called DT category, locally isomorphic to (a certain twist of) the matrix factorization category.
 For a quasi-smooth derived scheme $\bs{Y}$ and a conical closed subset $\mathcal{Z} \subset \bfT^*[-1]\bs{Y}$, Toda \cite[\S 3]{Tod19} proposes a definition of the $\bC^*$-equivariant version of the DT-category as a dg-quotient
 \begin{equation}\label{eq:catdimred}
  \mathcal{DT}^{\bC^*}(\bfT^*[-1]\bs{Y} \setminus \mathcal{Z}) \coloneqq D^b_{\mathrm{coh}}(\bs{Y}) / {\cC_{\cZ}}
\end{equation}
where $\cC_{\cZ}$ is the full dg-subcategory of $D^b_{\mathrm{coh}}(\bs{Y})$ spanned by objects whose singular support is contained in $\mathcal{Z}$. We refer the reader to \cite{AG15} for the definition of the singular support.

This definition is motivated by the Koszul duality equivalence stated as follows (see \cite[Theorem 2.3.3]{Tod19} for this version):
Let $U$ be a smooth scheme, $E$ be a vector bundle on $U$, and $s \in \Gamma(U, E)$ be a section.
We let $\bs{Z}(s)$ denote the derived zero locus of $s$ and $\bar{s} \colon E^\vee \to \bA^1$ the regular function corresponding to $s$.
Then we have an equivalence
\[
D_{\mathrm{coh}}^b(\bs{Z}(s)) \simeq \mathrm{MF}^{\bC^*}_{\mathrm{coh}}(U, \bar{s}).
\]
The right-hand side is the $\bC^*$-equivariant matrix factorization category where $\bC^*$ acts on the fiber of $E \to U$ by weight two.
This equivalence can be regarded as a categorification of Davison's dimensional reduction theorem \cite[Theorem A.1]{Dav17}.
Indeed, we have following isomorphisms:
\begin{align*}
\mathrm{HP}_*(D^b_{\mathrm{coh}}(\bs{Z}(s))) &\cong \HBM_*(t_0(\bs{Z}(s))) \otimes_{\bQ} \bC(\!(u)\!) \\
\mathrm{HP}_*(\mathrm{MF}^{\bC^*}_{\mathrm{coh}}(U, \bar{s})) &\cong \mH^{* + \vdim \bs{Z}(s)}(\bar{s} \, ^{-1}(0), \varphi_{\bar{s}}) \otimes_{\bQ} \bC(\!(u)\!)
\end{align*}
where $\mathrm{HP}$ denotes the periodic cyclic homology and $\bC(\!(u)\!)$ is the Laurent function field with the formal variable $u$ with degree $2$.
The first isomorphism follows from a result of Preygel \cite[Theorem 6.3.2]{Pre15} and the second isomorphism follows from a result of Efimov \cite[Theorem 1.1]{Efi18} (see also \cite[Lemma 3.3.2]{Tod19}).
Similarly, the definition \eqref{eq:catdimred} can be regarded as a categorification of the global dimensional reduction isomorphism \cite[Theorem 3.1]{Kin21}.

Now we can discuss a categorification of Theorem \ref{thm:main}.
As we do not know how to categorify the complex $\pi_! \varphi_{\bfT^*[-1] \bs{Y}}$ directly, we give a slightly different (though equivalent) version of Theorem \ref{thm:main}.
Let $0_Y \colon  Y = t_0(\bs{Y}) \hookrightarrow \widetilde{Y} = t_0(\bfT^*[-1]\bs{Y})$ be the zero section, and consider the following composition
\begin{equation}\label{eq:catmainthm}
R\Gamma(Y,\bQ_{Y}[\vdim \bs{Y}]) \cong  R\Gamma(Y, 0_Y^! \varphi_{\bfT^*[-1]\bs{Y}}) \to R\Gamma(\widetilde{Y},  \varphi_{\bfT^*[-1]\bs{Y}}) \cong R\Gamma(Y, \omega_{Y}[- \vdim \bs{Y}])
\end{equation}
where the first and third isomorphism follows from \cite[Theorem 3.1]{Kin21} and the second isomorphism is the $!$-counit map.
It is clear that the above composition is equal to the map given by \eqref{eq:pervVFC}.
The categorification of the complex 
\[
R\Gamma(Y, 0_Y^! \varphi_{\bfT^*[-1]\bs{Y}}) \simeq \Cone(R\Gamma(\widetilde{Y},  \varphi_{\bfT^*[-1]\bs{Y}}) \to R\Gamma(\widetilde{Y} \setminus Y,  \varphi_{\bfT^*[-1]\bs{Y}}|_{\widetilde{Y} \setminus Y} )) [1]
\]
is $\mathcal{C}_{0_Y(Y)}$ and the categorification of
the complex $R\Gamma(\widetilde{Y},  \varphi_{\bfT^*[-1]\bs{Y}})$ is
$\mathcal{DT}^{\bC^*}(\bfT^*[-1]\bs{Y} )$.
The categorification of the first isomorphism in \eqref{eq:catmainthm} is
\[
\Perf(\bs{Y}) \simeq \cC_{0_Y(Y)}
\]
which follows from \cite[Theorem 4.2.6]{AG15},
and the categorification of the latter isomorphism in \eqref{eq:catmainthm} is
\[
\mathcal{DT}^{\bC^*}(\bfT^*[-1]\bs{Y} ) \simeq D^b_{\mathrm{coh}}(\bs{Y})
\]
which is nothing but the definition.
Therefore the categorification of Theorem \ref{thm:main} is a tautological statement that the image of the structure sheaf $\cO_{\bs{Y}}$ under the inclusion
\[
\Perf(\bs{Y}) \hookrightarrow D_{\mathrm{coh}}^b(\bs{Y})
\]
is also $\cO_{\bs{Y}}$.

\begin{que}
Can we give another proof of Theorem \ref{thm:main} based on the above discussion combined with the virtual Riemann--Roch theorem \cite{FG10, CFK09}?
\end{que}

\begin{que}
Can we formulate a categorification of Conjecture \ref{conj:DT4} and define a version of ``DT4 virtual structure sheaf’’?
The first difficulty is defining a categorification of the perverse sheaf $\varphi_q(\varphi_{\bfT^*[-1]}\bs{Y})$ associated with a $-2$-shifted symplectic derived scheme $\bs{Y}$.
\end{que}


\section{Postponed proofs}

In this section, we will give proofs of Proposition \ref{prop:nearbyid}, Proposition \ref{prop:FSpure}, Proposition \ref{prop:FSunit}, and Proposition \ref{prop:dimredcomp}.

\subsection{Proof of Proposition \ref{prop:nearbyid}}\label{ssec:nearbyid}

First consider the following composition:
\begin{equation}\label{eq:nearbyid}
   \bQ_{0}[-2] \cong \Cone((\bC^* \hookrightarrow \bC)_! \bQ_{\bC^*} \to \bQ_{\bC})[-2] \to K'_{\psi} \to \bQ_{\bC}.
\end{equation}
Here, the second map (resp. the third map) is defined using the short exact sequence \eqref{eq:shex2} (resp. the short exact sequence \eqref{eq:shex1}).
We will prove that this is equal to the counit map
\[
 \bQ_{0}[-2] \cong (\{ 0 \} \hookrightarrow \bC)_! (\{ 0 \} \hookrightarrow \bC)^! \bQ_{\bC} \to \bQ_{\bC}
\]
later, and now prove the proposition assuming this statement.
Consider the following diagram
\[
\xymatrix{
i^* \sHom(\bQ_{Z\times \bC}, \pi^* \cF) \ar[r] \ar[d]^-{\simd} & i^* \sHom(i_! \bQ_{Z\times \{ 0 \}}[-2], \pi^* \cF) \ar[d]^-{\simd}  \\
i^* \pi^* \cF \ar[r]& i^! \pi^* \cF[2]
}
\]
where the upper horizontal arrow is induced from the counit map $i_! i^! \to \id$ and the bottom horizontal arrow is the natural map \eqref{eq:smpurity}. It is clear that this diagram commutes.
Therefore what we need to prove is that the following composition
\[
\cF \cong i^* \pi^* \cF \to i^! \pi^* \cF[2] \cF \to i^! \pi^! \cF \cong \cF
\]
where the second and third maps are the natural maps \eqref{eq:smpurity} is the identity map, but this is obvious since the following composition
\[
\bQ_Z \cong i^! \pi^! \bQ_Z \cong i^* \pi^!\bQ_Z[-2] \cong i^* \pi^* \bQ_Z \cong \bQ_Z
\]
is the identity map.

Now we prove that \eqref{eq:nearbyid} is given by the counit map.
To prove this statement, it is enough to prove the following composition is the identity map:
\[
\bQ \cong R\Gamma_c(\bQ_0) \to R\Gamma_c(K'_{\psi})[2] \to R\Gamma_c(\bQ_{\bC}[2]) \cong \bQ.
\]
It is clear that the above map is given by
\[
\bQ \cong R\Gamma_c ([p_! \bQ_{\bCu} \xrightarrow{1 - T} p_! \bQ_{\bCu} \xrightarrow{\mathrm{tr}} \bQ_{\bC}]) \xrightarrow[]{- \pr} R\Gamma_c (p_! \bQ_{\bCu}[2]) \cong \bQ.
\]
where $\pr$ denotes the canonical projection.

From now we work with $\bR$-coefficient, which does not affect the conclusion.
We compute the above composition using differential forms.
Recall that for a differentiable manifold $M$, we can take the de Rham resolution of a constant sheaf
\[
\bR_M \simeq \Omega_M ^\bullet = [\Omega_M^0 \xrightarrow{d} \Omega_M^1 \xrightarrow{d} \cdots].
\]
Consider the following resolution
\[
[p_! \bR_{\bCu} \xrightarrow{1 - T} p_! \bR_{\bCu} \xrightarrow{\mathrm{tr}} \bR_{\bC}] \simeq
\Tot([p_! \Omega ^{\bullet}_{\bCu} \xrightarrow{1 - \alpha^*} p_! \Omega ^{\bullet}_{\bCu} \xrightarrow{p_*} \Omega ^{\bullet}_{\bC}])
\]
where $\Tot$ denotes the totalization and $\alpha$ is the covering transformation corresponding to a counterclockwise loop.
Since all sheaves appearing on the right-hand side are c-soft, the following complex
\[
C^\bullet \coloneqq \Gamma_c(\Tot([p_! \Omega ^{\bullet}_{\bCu} \xrightarrow{1 - \alpha^*} p_! \Omega ^{\bullet}_{\bCu} \xrightarrow{p_*} \Omega ^{\bullet}_{\bC}]))
\]
computes the derived functor, hence isomorphic to $\bR$ concentrated in degree zero.
Note that $C^0$ consists of triples of differential forms $(f, \omega_1, \omega_2) \in \Omega^0_{\bC} \oplus \Omega^1_{\bCu} \oplus \Omega^2_{\bCu}$ with compact supports such that
\[
df = p_* \omega_1,\  -d \omega_1 = \omega_2 - \alpha^* \omega_2.
\]
Now what we need to prove is
\[
f(0) = - \int_{\bCu} \omega_2
\]
for such a triple, but this is an easy consequence of Stokes's theorem.

\subsection{Proof of Proposition \ref{prop:FSpure}}\label{ssec:FSpure}

We use the notation as in the diagram \eqref{eq:diagFS}.
 Firstly for a complex vector bundle $E$, we define two functors
 \[
 \FSn_E, \FSn'_E \colon \DR^+(E) \to \DR^+(E^\vee)
 \]
 by 
  \begin{align*}
    \FSn_E &\coloneqq p'_! {\iota_{P'}}_* \iota_{P'}^* {\iota_{P}}_! \iota_{P}^!  p^* \\
    \FSn_E' &\coloneqq p'_* {\iota_{P'}}_* \iota_{P'}^* {\iota_{P}}_! \iota_{P}^!  p^*.
  \end{align*}
  Then we have the following commutative diagram:
  \[
  \xymatrix{
  \FSn_E  \ar[r]^-{\sim} \ar[d]^-{\simu} & \FSn_E'  \\
  \FS_E \ar[r]^-{\sim} & \FS_E' \ar[u]_-{\simd}
  }
  \]
where the upper horizontal arrow is induced from the natural transform $p'_! \to p'_*$,
the left vertical arrow is induced from the counit map ${\iota_{P}}_! \iota_{P}^! \to \id$ and
the right vertical arrow is induced from the unit map $\id \to {\iota_{P'}}_* \iota_{P'}^*$.

By repeating the construction of \eqref{eq:FSbcc1} and \eqref{eq:FSbcc2},
we can construct natural isomorphisms
\begin{align*}
\FSn_{E_2}(g_! \cF) &\cong {}^t g^* \FSn_{E_1}(\cF)   \\
\FSn_{E_2}'(g_* \cF) &\cong {}^t g^! \FSn_{E_1}'(\cF) [2 \dim g]
\end{align*}
by composing base change maps or the inverse of base change maps.
These maps are compatible with the natural transforms $\FS_{E} \cong \FSn_{E}$ and $\FS_{E}' \cong \FSn_{E}'$
since the $!$-counit and $*$-unit commute with base change maps as we have seen in \S \ref{sect:nattrans}.
Therefore we need to prove the commutativity of the following diagram:
\[
\xymatrix{
\FSn_{E_2}(g_! \cF)\ar[d]^-{\simd} \ar[r] & \FSn'_{E_2}(g_* \cF) \ar[d]^-{\simd} \\
{}^t g^*\FSn_{E_1}(\cF) \ar[r] & {}^t g^!\FSn_{E_1}'(\cF)[-2 \dim ^t g].
}
\]
By definition the left vertical map is decomposed into following isomorphisms
\begin{align*}
  \FSn_{E_2}(g_! \cF) &= {p'_2}_! {\iota_{P'_2, *}} \iota_{P'_2}^* {\iota_{P_2, !}} \iota_{P_2}^!  p^*_2 g_! \cF \\
                      &\xleftarrow{\sim}  \widetilde{p_1'}_! {\iota_{\widetilde{P'}, *}} \iota_{\widetilde{P'}}^* {\iota_{\widetilde{P}, !}} \iota_{\widetilde{P}}^! \widetilde{^t g}^* p_1^* \cF \\
                      &\xleftarrow{\sim} {}^t g^* {p_1'}_! {\iota_{P_1', *}} \iota_{P_1'}^* {\iota_{P_1, !}} \iota_{P_1}^! p_1^* \cF  \\
                      &= {}^t g^*\FSn_{E_1}(\cF)
\end{align*}
and the right vertical map is decomposed into the following isomorphisms
\begin{align*}
\FSn_{E_2}'(g_* \cF) &= {p'_2}_* {\iota_{P'_2}}_* \iota_{P'_2}^* {\iota_{P_2}}_! \iota_{P_2}^! p_2^* g_* \cF \\
                     &\xrightarrow{\sim}  \widetilde{p_1'}_* {\iota_{\widetilde{P'}, *}} \iota_{\widetilde{P'}}^* {\iota_{\widetilde{P}, !}} \iota_{\widetilde{P}}^!  \widetilde{^t g}^* p_1^* \cF \\
                     &\cong   \widetilde{p_1'}_* {\iota_{\widetilde{P'}, *}} \iota_{\widetilde{P'}}^* {\iota_{\widetilde{P}, !}} \iota_{\widetilde{P}}^! \widetilde{^t g}^! p_1^* \cF [-2 \dim ^t g] \\
                     &\xrightarrow{\sim} {}^t g^! {p_1'}_* {\iota_{P_1', *}} \iota_{P_1'}^* {\iota_{P_1, !}} \iota_{P_1}^!  p_1^* \cF  [-2 \dim ^t g] \\
                     &= {}^t g^!\FSn_{E_1}'(\cF)[-2 \dim ^t g].
\end{align*}
Then the claim follows from the fact that the natural transformations \eqref{eq:!*1} and \eqref{eq:smpurity} commute with base change maps and the commutativity of the diagrams \eqref{eq:puritybcc}.

\subsection{Proof of Proposition \ref{prop:FSunit}}\label{ssec:FSunit}

To prove this statement, we need several lemmas.
\begin{lem}\label{lem:FSeta_*}
  The following diagram commutes up to the sign $(-1)^{\rank E_1}$:
  \begin{equation}\label{eq:FSeta_*}
  \xymatrix{
  \FS_{E_2^\vee} \FS_{E_2} (g_* \cF) \ar[r]^-{\sim} \ar[d]^-{\simd}_-{ \eta(g_* \cF)} & \FS_{E^\vee _2} {}^t g^! \FS_{E_1} ( \cF)[2 \dim g] \ar[r]^-{\sim} & g_* \FS_{E_1^\vee} \FS_{E_1} ( \cF) [2 \dim g] \ar[d]^-{\simd}_-{g_* \eta(\cF)} \\
  g_* \cF^a[-2 \rank E_2] \ar@{=}[rr] & & g_* \cF^a[-2 \rank E_2].
  }
\end{equation}
\end{lem}
\begin{proof}
  Note that the map $\FS_{E^\vee _2} {}^t g^! \FS_{E_1} ( \cF) \to  g_* \FS_{E_1^\vee} \FS_{E_1} ( \cF)$ is defined by the composition
  \begin{align*}
  \FS_{E^\vee _2} {}^t g^! \FS_{E_1} ( \cF) &\xrightarrow[\sim]{\FS_{E^\vee _2} {}^t g^!(\eta(\FS_{E_1} ( \cF^a))^{-1})} \FS_{E^\vee _2} {}^t g^! \FS_{E_1} \FS_{E_1^\vee} \FS_{E_1} ( \cF^a)[2 \rank E_1] \\
            &\xrightarrow[\sim]{}  \FS_{E^\vee _2}  \FS_{E_2} g_* \FS_{E_1^\vee} \FS_{E_1} ( \cF^a)[2 \rank E_2] \\
                                            &\xrightarrow[\sim]{\eta(g_* \FS_{E_1^\vee} \FS_{E_1} ( \cF^a))} g_* \FS_{E^\vee _1}  \FS_{E_1}  (\cF).
  \end{align*}
We have seen that $\eta(\FS_{E_1} ( \cF^a))$ differs from $\FS_{E_1}(\eta(\cF^a)) $ by $(-1)^{\rank E_1}$ in \eqref{eq:FSetaa}. Then the claim follows immediately.
\end{proof}
Take an object $\cG \in \DR^+(E_2)$. A similar argument shows that the following three diagrams commute up to the sign $(-1)^{\rank E_1}$:
\begin{equation}\label{eq:FSeta_!}
\xymatrix{
\FS_{E_2^\vee} \FS_{E_2} (g_! \cF) \ar[r]^-{\sim} \ar[d]^-{\simd}_-{ \eta(g_! \cF)} & \FS_{E^\vee _2} {}^t g^* \FS_{E_1} ( \cF) \ar[r]^-{\sim} & g_! \FS_{E_1^\vee} \FS_{E_1} ( \cF) [2 \dim g] \ar[d]^-{\simd}_-{g_! \eta(\cF)} \\
g_! \cF^a[-2 \rank E_2] \ar@{=}[rr] & & g_! \cF^a[-2 \rank E_2].
}
\end{equation}
\begin{equation}\label{eq:FSeta^*}
\xymatrix{
\FS_{E_1^\vee} \FS_{E_1} (g^* \cG) \ar[r]^-{\sim} \ar[d]^-{\simd}_-{ \eta(g^* \cG)} & \FS_{E^\vee _1} {}^t g_! \FS_{E_2} ( \cG)[- 2 \dim g] \ar[r]^-{\sim} & g^* \FS_{E_2^\vee} \FS_{E_2} ( \cG) [- 2 \dim g] \ar[d]^-{\simd}_-{g^* \eta(\cG)} \\
g^* \cG^a[-2 \rank E_1] \ar@{=}[rr] & & g^* \cG^a[-2 \rank E_1].
}
\end{equation}
\begin{equation}\label{eq:FSeta^!}
\xymatrix{
\FS_{E_1^\vee} \FS_{E_1} (g^! \cG) \ar[r]^-{\sim} \ar[d]^-{\simd}_-{ \eta(g^! \cG)} & \FS_{E^\vee _1} {}^t g_* \FS_{E_2} ( \cG) \ar[r]^-{\sim} & g^! \FS_{E_2^\vee} \FS_{E_2} ( \cG) [- 2 \dim g] \ar[d]^-{\simd}_-{g^! \eta(\cG)} \\
g^! \cG^a[-2 \rank E_1] \ar@{=}[rr] & & g^! \cG^a[-2 \rank E_1].
}
\end{equation}

\begin{lem}
The following diagram commutes up to the sign $(-1)^{\dim g}$:
\[
\xymatrix{
\cF^a[-2 \rank E_1] \ar[r] \ar[d]^-{\simd} & g^! g_! \cF^a[-2 \rank E_1] \ar[dd]^-{\simd}  \\
\FS_{E_1^\vee} \FS_{E_1}(\cF) \ar[d] & \\
\FS_{E_1^\vee} {}^t g_* {}^t g^* \FS_{E_1}(\cF) \ar[r]^-{\sim} & g^! \FS_{E_2^\vee} \FS_{E_2}(g_! \cF)[-2 \dim g].
}
\]
\end{lem}
\begin{proof}
  Using the fact that the diagram \eqref{eq:FSeta} commutes up to the sign $(-1)^{\rank E}$, we need to prove the commutativity of the following diagram:
  \[
  \xymatrix{
  \cF \ar[r] \ar[d]^-{\simd} & g^! g_! \cF \ar[dd]^-{\simd}  \\
  \FSS_{E_1^\vee} \FS_{E_1}(\cF) \ar[d] & \\
  \FSS_{E_1^\vee} {}^t g_* {}^t g^* \FS_{E_1}(\cF) \ar[r]^-{\sim} & g^! \FSS_{E_2^\vee} \FS_{E_2}(g_! \cF).
  }
  \]
  Here the natural transformation
  \[
  \FSS_{E_1^\vee} \circ {}^t g_* \xrightarrow{\sim} g^! \circ \FSS_{E_2^\vee}
  \]
  is defined by composing base change maps as in the construction of the map \eqref{eq:FSbcc2}.
  Note that all arrows in the diagram except for the horizontal bottom arrow are defined by composing unit maps.
  Therefore the claim follows from the construction of the base change map \eqref{eq:bc5} in \cite[Proposition 3.1.9(ii)]{KS13}.
\end{proof}

The following statement is a consequence of the above lemma and the commutativity of the diagram \eqref{eq:FSeta}, \eqref{eq:FSeta_*} and \eqref{eq:FSeta^*}
up to a certain choice of sign.

\begin{cor}\label{cor:1}
  The following diagram commutes:
  \[
  \xymatrix{
  \cG^a[-2 \rank E_2] \ar[r] \ar[d]^-{\simd} & g_* g^* \cG^a[-2 \rank E_2] \ar[dd]^-{\simd}  \\
  \FS_{E_2^\vee} \FS_{E_2}(\cG) \ar[d] & \\
  \FS_{E_2^\vee} {}^t g^! {}^t g_! \FS_{E_2}(\cG) \ar[r]^-{\sim} & g_* \FS_{E_1^\vee} \FS_{E_1}(g^* \cG)[2 \dim g].
  }
  \]
\end{cor}

The following two statements can be proved in a similar way:

\begin{lem}
The following diagram commutes:
\[
\xymatrix{
g^* g_* \cF^a[-2 \rank E_1] \ar[r] &  \cF^a[-2 \rank E_1]   \\
  & \FS_{E_1^\vee} \FS_{E_1}(\cF) \ar[u]_-{\simd} \\
g^* \FS_{E_2^\vee} \FS_{E_2} (g_* \cF)[-2 \dim g] \ar[r]^-{\sim} \ar[uu]_-{\simd} & \FS_{E_1^\vee} {}^t g_! {}^t g^! \FS_{E_1}(\cF). \ar[u]
}
\]
\end{lem}

\begin{cor}
  The following diagram commutes up to the sign $(-1)^{\dim g}$:
  \[
  \xymatrix{
  g_! g^! \cG^a[-2 \rank E_2] \ar[r] &  \cG^a[-2 \rank E_2]   \\
    & \FS_{E_2^\vee} \FS_{E_2}(\cG) \ar[u]_-{\simd} \\
  g_! \FS_{E_1^\vee} \FS_{E_1} (g^! \cG^a)[2 \dim g] \ar[r]^-{\sim} \ar[uu]_-{\simd} & \FS_{E_2^\vee} {}^t g^* {}^t g_* \FS_{E_2}(\cG). \ar[u]
  }
  \]
\end{cor}

\begin{proof}[Proof of Proposition \ref{prop:FSunit}]
  We can prove the commutativity of these diagrams up to a certain choice of sign in a similar way, so we only prove that the left diagram commutes up to the sign $(-1)^{\rank E_1}$.
  Using Corollary \ref{cor:1}, we need to prove that the following diagram commutes up to the sign $(-1)^{\rank E_1}$:
  \[
  \xymatrix{
  \FS_{E_2} \cG^a[-2 \rank E_2] \ar[r] \ar[d]^-{\simd}_-{\FS_{E_2}\eta(\cG)^{-1}} & {}^t g^! {}^t g_! \FS_{E_2}  (\cG^a)[-2 \rank E_2] \ar[d]^-{\simd}  \\
  \FS_{E_2} \FS_{E_2^\vee} \FS_{E_2}(\cG) \ar[d] & \FS_{E_2} (g_* g^* \cG^a)[-2 \rank E_2] \ar[d]^-{\simd} \\
  \FS_{E_2} \FS_{E_2^\vee} {}^t g^! {}^t g_! \FS_{E_2}(\cG) \ar[r]^-{\sim} & \FS_{E_2} g_* \FS_{E_1^\vee} \FS_{E_1}(g^* \cG)[2 \dim g].
  }
  \]
  This follows using the commutativity of the diagram \eqref{eq:FSetaa} up to the sign $(-1)^{\rank E}$ twice and Lemma \ref{lem:FSeta_*}.
\end{proof}

\subsection{Proof of Proposition \ref{prop:dimredcomp}}\label{ssec:dimredcomp}

  We first prove the statement (i). To do this, consider the following diagram in $D^b_c(E_Z)$ :
  \begin{equation}\label{eq:zigzag1}
    \def\objectstyle{\scriptstyle}
\def\labelstyle{\scriptscriptstyle}
\xymatrix@C=15pt@R=15pt{
 & & & 0_{E_Z, *}(\cF|_Z) \\
 & \FS_{E_Z^\vee}((\pi_{E^\vee}^! \cF)|_{E_Z^\vee}) \ar[r]^-{\eqref{eq:FSbc^*}}_-{\sim} & \FS_{E^\vee}(\pi_{E^\vee}^! \cF)|_{E_Z} \ar[ru]^-{\eqref{eq:FSbcc3}}_-{\sim} & \iota_{C_{Z/X}, *} \Sp_{Z/X}(\cF_Z) \ar[u]^-{\eqref{eq:sp*}}_-{\simd} \\
 \FS_{E_Z^\vee}(\varphi_{\bar{s}}(\pi_{E^\vee}^! \cF)|_{E_Z^\vee}) \ar[rrr]^-{(-1) \cdot \eqref{eq:vansp}}_-{\sim} \ar[ru]^-{\eqref{eq:v*}} & & & \iota_{C_{Z/X}, *} \Sp_{Z/X}(\cF). \ar[u]
}
  \end{equation}
We claim that this diagram is commutative.
The most complicated map in this diagram is of course \eqref{eq:vansp} as it is defined by the composition of seven morphisms \eqref{eq:compvansp}. To deal with this map, consider the following composition ``parallel’’ to the composition \eqref{eq:compvansp}:
\begin{align}\label{eq:compvanspx}
  \begin{split}
    \FS_{E^\vee_Z}((\pi_{E^\vee}^!\cF) |_{E_Z^\vee})
    & \xrightarrow[]{\sim} \FS_{E^\vee_Z}( (\hat{s}^!  p_{\bA^1}^! \cF) |_{E_Z^\vee}) \\
    & \xrightarrow[\eqref{eq:FSbc^*}]{\sim} \FS_{E^\vee}( \hat{s}^! p_{\bA^1}^! \cF) |_{E_Z} \\
    & \xrightarrow[\eqref{eq:FSbcc3}]{\sim} (\tilde{s}_* \FS_{X \times \bA^1}(  p_{\bA^1}^! \cF)) |_{E_Z} \\
    & \xrightarrow[\eqref{eq:FSbcc3}]{\sim} (\tilde{s}_* 0_{X/{X \times \bA^1}, *} \cF)|_{E_Z} \\
    & \cong 0_{E_Z, *}(\cF|_Z) \\
    & = 0_{E_Z, *}(\cF|_Z) \\
    & = 0_{E_Z, *}(\cF|_Z)
  \end{split}
\end{align}
where $0_{X/{X \times \bA^1}} \colon X \hookrightarrow X\times \bA^1$ in the fourth line is the zero section.
Here what we mean by this composition is ``parallel’’ to the composition \eqref{eq:compvansp} is that each complex appearing in \eqref{eq:compvansp} has a natural map to the complex appearing in the above composition at the same stage.
It is clear that these morphisms commute with morphisms appearing in \eqref{eq:compvansp} and \eqref{eq:compvanspx} except for the ones in the fourth rows.
To prove the commutativity of \eqref{eq:zigzag1} we need to show that the maps in the fourth rows of \eqref{eq:compvansp} and \eqref{eq:compvanspx} commute with the natural maps from the complexes at the third and fourth rows of \eqref{eq:compvansp} to the complexes at the same stages of \eqref{eq:compvanspx} up to the sign $-1$, i.e. we need to show the commutativity of the following diagram:
\begin{align}
\begin{split}
  \xymatrix{
   \FS_{X \times \bA^1}(  i_{X \times \ell_{\leq 0}, *}  (
  p_{\ell_ {\leq 0}}^! \cF)) \ar[r] \ar[d]^-{\eqref{eq:FShalf}}_-{\simd}  &  \FS_{X \times \bA^1}(  p_{\bA^1}^! \cF) \ar[d]^-{(-1) \cdot \eqref{eq:FSbcc3}}_-{\simd}  \\
   j_{X \times \bA^1_{>0}, *} p_{\bA^1_{ >0}}^* \cF \ar[r] & 0_{X/{X \times \bA^1}, *} \cF.
  }
\end{split}
\end{align}

To do this, consider the following  larger diagram:
\begin{equation}\label{eq:largerectangle}
\begin{split}
  \def\objectstyle{\scriptstyle}
\def\labelstyle{\scriptscriptstyle}
  \xymatrix{
  \FS_{X \times \bA^1}(0_{X/{X \times \bA^1}, *} \cF) \ar[r] \ar[d]^-{\eqref{eq:FSbcc1}}_-{\simd} 
  & \FS_{X \times \bA^1}(  i_{X \times \ell_{\leq 0}, *}  p_{\ell_ {\leq 0}}^! \cF) \ar[r] \ar[d]^-{\eqref{eq:FShalf}}_-{\simd}  
  &  \FS_{X \times \bA^1}(  p_{\bA^1}^! \cF) \ar[d]^-{(-1) \cdot \eqref{eq:FSbcc3}}_-{\simd} \\
  p_{\bA^1}^* \cF \ar[r]  
  & j_{X \times \bA^1_{>0}, *} p_{\bA^1_{ >0}}^* \cF \ar[r] & 0_{X/{X \times \bA^1}, *} \cF.
  }
\end{split}
\end{equation}
The commutativity of the left rectangle is clear from the construction of the maps \eqref{eq:FSbcc1} and \eqref{eq:FShalf} and the commutativity of the outer rectangle follows from Proposition \ref{prop:FSunit}.
As the lower horizontal arrows become isomorphic after restricting to the zero section $X \times \{ 0 \} \subset X \times \bA^1$, we conclude the commutativity of the right rectangle.
Now the proof of the commutativity of \eqref{eq:zigzag1} is over.

Next consider the following diagram in $D_c^b(Z)$:
\begin{equation}\label{eq:zigzag2}
  \def\objectstyle{\scriptstyle}
\def\labelstyle{\scriptscriptstyle}
\xymatrix@C=-10pt@R=15pt{
\pi_{E_Z^\vee, !}(\varphi_{\bar{s}}(\pi_{E^\vee}^! \cF) |_{E_Z^\vee}) \ar[dd]^-{\eqref{eq:FSbcc1}}_-{\simd} \ar[rr]^-{\eqref{eq:v*}}   &  & \pi_{E_Z^\vee, !}((\pi_{E^\vee}^! \cF) |_{E_Z^\vee}) \ar[rr] \ar@{-->}[ld]_-{\eqref{eq:FSbcc1}}  & & \cF|_Z  \\
& 0_{E_Z}^* \FS_{E_Z^\vee}((\pi_{E^\vee}^! \cF)|_{E_Z^\vee}) \ar[rr]^-{\eqref{eq:FSbc^*}}_-{\sim} &  &
0_{E_Z}^* (\FS_{E^\vee}(\pi_{E^\vee}^! \cF)|_{E_Z}) \ar[ru]^-{\eqref{eq:FSbcc3}}_-{\sim} & \\
0_{E_Z}^* \FS_{E_Z^\vee}(\varphi_{\bar{s}}(\pi_{E^\vee}^! \cF)|_{E_Z^\vee})  \ar[ru]^-{\eqref{eq:v*}}. &  & & &
}
\end{equation}
The commutativity of the left quadrilateral is clear and the commutativity of the right quadrilateral follows from the commutativity of
the diagram \eqref{eq:FSbccFSbc} and Proposition \ref{prop:FSunit}.

Finally combining the commutativity of the diagram \eqref{eq:zigzag1} restricted to the zero section $Z \subset E_Z$ and the commutativity of the diagram \eqref{eq:zigzag2}, we obtain the commutativity of the diagram \eqref{eq:dimredcomp1}.

Now we move on to the proof of the statement (ii).
The proof is similar to the proof of the statement (i).
Consider the following diagram in $D^b_c(E_Z)$:
\begin{equation}\label{eq:zigzag3}
  \def\objectstyle{\scriptstyle}
\def\labelstyle{\scriptscriptstyle}
\xymatrix@C=15pt@R=15pt{
& & & 0_{E_Z, *}(\cF|_Z ^!) \ar[d]_-{\eqref{eq:sp!}}^-{\simd} \\
& \FS_{E_Z^\vee}((\pi_{E^\vee}^! \cF)|_{E_Z^\vee}^!) \ar[ld]_-{\eqref{eq:!v}} & \FS_{E^\vee}(\pi_{E^\vee}^! \cF)|_{E_Z}^! \ar[ru]^-{\eqref{eq:FSbcc3}}_-{\sim} \ar[l]_-{\eqref{eq:FSbc^!}}^-{\sim} & \iota_{C_{Z/X}, *} \Sp_{Z/X}(\Gamma_Z(\cF)) \ar[d] \\
\FS_{E_Z^\vee}(\varphi_{\bar{s}}(\pi_{E^\vee}^! \cF)|_{E_Z^\vee}^!) \ar[rrr]^-{(-1) \cdot \eqref{eq:vansp}}_-{\sim}  & & & \iota_{C_{Z/X}, *} \Sp_{Z/X}(\cF).
}
\end{equation}
We will prove the commutativity of this diagram in the same way as the proof of the commutativity of the diagram \eqref{eq:zigzag1}.
Consider the following composition:

\begin{align}\label{eq:compvanspxx}
  \begin{split}
    \FS_{E^\vee_Z}((\pi_{E^\vee}^!\cF) |_{E_Z^\vee}^!)
        & \xrightarrow[]{\sim} \FS_{E^\vee_Z}( (\hat{s}^! 0_{X/{X\times \bA^1}, *} \cF) |_{E_Z^\vee}^!) \\
        & \xrightarrow[\eqref{eq:FSbc^!}]{\sim} \FS_{E^\vee_Z}( \hat{s}^! 0_{X/{X\times \bA^1}, *} \cF) |_{E_Z}^! \\
        & \xrightarrow[\eqref{eq:FSbcc3}]{\sim} (\tilde{s}_* \FS_{X \times \bA^1}(  0_{X/{X\times \bA^1}, *} \cF)) |_{E_Z}^! \\
        & \xrightarrow[\eqref{eq:FSbcc1}]{\sim} (\tilde{s}_* p_{\bA^1}^* \cF) |_{E_Z}^! \\
        & \cong 0_{E_Z, *}(\cF|_Z^!) \\
        & = 0_{E_Z, *}(\cF|_Z^!) \\
        & = 0_{E_Z, *}(\cF|_Z^!)
  \end{split}
\end{align}
We have natural maps from the complexes appearing in \eqref{eq:compvanspxx} to the complexes appearing in \eqref{eq:compvansp} at the same stages such that these maps commute with maps appearing in \eqref{eq:compvanspxx} and \eqref{eq:compvansp}.
Furthermore, the commutativity of the outer rectangle of the diagram \eqref{eq:largerectangle} shows that the composition \eqref{eq:compvanspxx} is equal to the following composition after multiplying by $-1$:
\[
\FS_{E_Z^\vee}((\pi_{E^\vee}^! \cF)|_{E_Z^\vee}^!) \xleftarrow[\sim]{\eqref{eq:FSbc^!}} \FS_{E^\vee}(\pi_{E^\vee}^! \cF)|_{E_Z}^! \xrightarrow[\sim]{\eqref{eq:FSbcc3}} 0_{E_Z, *}(\cF|_Z ^!)
\]
These claims imply the commutativity of the diagram \eqref{eq:zigzag3}.

Next consider the following diagram:
\begin{equation}\label{eq:zigzag4}
  \def\objectstyle{\scriptstyle}
\def\labelstyle{\scriptscriptstyle}
\xymatrix@C=-25pt@R=15pt{
\pi_{E_Z^\vee, *}\varphi_{\bar{s}}(\pi_{E^\vee}^! \cF) |_{E_Z^\vee}[-2 \rank E] \ar[dd]^-{\eqref{eq:FSbcc2}}_-{\simd}    
&  
& \pi_{E_Z^\vee, *}(\pi_{E^\vee}^! \cF )|_{E_Z^\vee}^! [-2 \rank E] \ar[ll]_-{\eqref{eq:!v}}  \ar@{-->}[ld]_-{\eqref{eq:FSbcc2}}  & & \cF|_Z^! \ar[ll] \\
& 0_{E_Z}^! \FS_{E_Z^\vee}((\pi_{E^\vee}^! \cF)|_{E_Z^\vee}^!) \ar[ld]_-{\eqref{eq:!v}}  &  & 0_{E_Z}^! \FS_{E^\vee}(\pi_{E^\vee}^! \cF)|_{E_Z}^! \ar[ru]^(.3){(-1)^{\rank E} \cdot \eqref{eq:FSbcc3}}_-{\sim} \ar[ll]_-{\eqref{eq:FSbc^!}}^-{\sim} & \\
0_{E_Z}^! \FS_{E_Z^\vee}(\varphi_{\bar{s}}(\pi_{E^\vee}^! \cF)|_{E_Z^\vee}^!) . &  & & &
}
\end{equation}
The commutativity of the left quadrilateral is obvious, and the commutativity of the right quadrilateral follows from Proposition \ref{prop:FSpure} and Proposition \ref{prop:FSunit}.

Finally combining the commutativity of the diagrams \eqref{eq:zigzag3} and \eqref{eq:zigzag4}, we obtain the commutativity of the diagram \eqref{eq:dimredcomp2}.

\bibliographystyle{plain}
\bibliography{biblio}

\begin{thebibliography}{10}

\bibitem{AB17}
Lino Amorim and Oren Ben-Bassat.
\newblock Perversely categorified {L}agrangian correspondences.
\newblock {\em Adv. Theor. Math. Phys.}, 21(2):289--381, 2017.

\bibitem{AG15}
D.~Arinkin and D.~Gaitsgory.
\newblock Singular support of coherent sheaves and the geometric {L}anglands
  conjecture.
\newblock {\em Selecta Math. (N.S.)}, 21(1):1--199, 2015.

\bibitem{Beh97}
K.~Behrend.
\newblock Gromov-{W}itten invariants in algebraic geometry.
\newblock {\em Invent. Math.}, 127(3):601--617, 1997.

\bibitem{BF97}
K.~Behrend and B.~Fantechi.
\newblock The intrinsic normal cone.
\newblock {\em Invent. Math.}, 128(1):45--88, 1997.

\bibitem{Beh09}
Kai Behrend.
\newblock {D}onaldson--{T}homas type invariants via microlocal geometry.
\newblock {\em Ann. of Math. (2)}, 170(3):1307--1338, 2009.

\bibitem{BBBBJ15}
Oren Ben-Bassat, Christopher Brav, Vittoria Bussi, and Dominic Joyce.
\newblock A `{D}arboux theorem' for shifted symplectic structures on derived
  {A}rtin stacks, with applications.
\newblock {\em Geom. Topol.}, 19(3):1287--1359, 2015.

\bibitem{BJ17}
Dennis Borisov and Dominic Joyce.
\newblock Virtual fundamental classes for moduli spaces of sheaves on
  {C}alabi-{Y}au four-folds.
\newblock {\em Geom. Topol.}, 21(6):3231--3311, 2017.

\bibitem{BBDJS15}
C.~Brav, V.~Bussi, D.~Dupont, D.~Joyce, and B.~Szendr\H{o}i.
\newblock Symmetries and stabilization for sheaves of vanishing cycles.
\newblock {\em J. Singul.}, 11:85--151, 2015.
\newblock With an appendix by J\"{o}rg Sch\"{u}rmann.

\bibitem{BBJ19}
Christopher Brav, Vittoria Bussi, and Dominic Joyce.
\newblock A {D}arboux theorem for derived schemes with shifted symplectic
  structure.
\newblock {\em J. Amer. Math. Soc.}, 32(2):399--443, 2019.

\bibitem{CL14}
Yalong Cao and Naichung~Conan Leung.
\newblock {D}onaldson--{T}homas theory for {C}alabi--{Y}au 4-folds.
\newblock {\em arXiv preprint arXiv:1407.7659}, 2014.

\bibitem{CFK09}
Ionu\c{t} Ciocan-Fontanine and Mikhail Kapranov.
\newblock Virtual fundamental classes via dg-manifolds.
\newblock {\em Geom. Topol.}, 13(3):1779--1804, 2009.

\bibitem{Dav17}
Ben Davison.
\newblock The critical {C}o{HA} of a quiver with potential.
\newblock {\em Q. J. Math.}, 68(2):635--703, 2017.

\bibitem{EG95}
Dan Edidin and William Graham.
\newblock Characteristic classes and quadric bundles.
\newblock {\em Duke Math. J.}, 78(2):277--299, 1995.

\bibitem{Efi18}
Alexander~I. Efimov.
\newblock Cyclic homology of categories of matrix factorizations.
\newblock {\em Int. Math. Res. Not. IMRN}, (12):3834--3869, 2018.

\bibitem{FG10}
Barbara Fantechi and Lothar G\"{o}ttsche.
\newblock Riemann--{R}och theorems and elliptic genus for virtually smooth
  schemes.
\newblock {\em Geom. Topol.}, 14(1):83--115, 2010.

\bibitem{Ful84}
William Fulton.
\newblock {\em Intersection theory}, volume~2 of {\em Ergebnisse der Mathematik
  und ihrer Grenzgebiete (3) [Results in Mathematics and Related Areas (3)]}.
\newblock Springer-Verlag, Berlin, 1984.

\bibitem{Joy15}
Dominic Joyce.
\newblock A classical model for derived critical loci.
\newblock {\em Journal of Differential Geometry}, 101(2):289--367, 2015.

\bibitem{Joyslide}
Dominic Joyce.
\newblock Shifted symplectic derived algebraic geometry and generalizations of
  {D}onaldson–{T}homas theory.
\newblock 2018.
\newblock URL:https://people.maths.ox.ac.uk/joyce/SeoulHandout3.pdf.

\bibitem{JU21}
Dominic Joyce and Markus Upmeier.
\newblock Orientation data for moduli spaces of coherent sheaves over
  {C}alabi-{Y}au 3-folds.
\newblock {\em Adv. Math.}, 381:Paper No. 107627, 47, 2021.

\bibitem{KS13}
Masaki Kashiwara and Pierre Schapira.
\newblock {\em Sheaves on Manifolds: With a Short History.{\guillemotleft}Les
  d{\'e}buts de la th{\'e}orie des faisceaux{\guillemotright}. By Christian
  Houzel}, volume 292.
\newblock Springer Science \& Business Media, 2013.

\bibitem{Kin21}
Tasuki {Kinjo}.
\newblock {Dimensional reduction in cohomological {D}onaldson--{T}homas
  theory}.
\newblock {\em to appear in Compos. Math.}

\bibitem{LT98}
Jun Li and Gang Tian.
\newblock Virtual moduli cycles and {G}romov-{W}itten invariants of algebraic
  varieties.
\newblock {\em J. Amer. Math. Soc.}, 11(1):119--174, 1998.

\bibitem{Mas16}
David Massey.
\newblock Natural commuting of vanishing cycles and the {V}erdier dual.
\newblock {\em Pacific J. Math.}, 284(2):431--437, 2016.

\bibitem{OT20}
Jeongseok Oh and Richard~P Thomas.
\newblock Counting sheaves on {C}alabi-{Y}au 4-folds, {I}.
\newblock {\em arXiv preprint arXiv:2009.05542}, 2020.

\bibitem{PTVV13}
Tony Pantev, Bertrand To\"{e}n, Michel Vaqui\'{e}, and Gabriele Vezzosi.
\newblock Shifted symplectic structures.
\newblock {\em Publ. Math. Inst. Hautes \'{E}tudes Sci.}, 117:271--328, 2013.

\bibitem{Pre15}
Anatoly Preygel.
\newblock Ind-coherent complexes on loop spaces and connections.
\newblock In {\em Stacks and categories in geometry, topology, and algebra},
  volume 643 of {\em Contemp. Math.}, pages 289--323. Amer. Math. Soc.,
  Providence, RI, 2015.

\bibitem{Sch22}
Kendric Schefers.
\newblock An equivalence between vanishing cycles and microlocalization.
\newblock {\em arXiv preprint arXiv:2205.12436}.

\bibitem{Sch03}
J\"{o}rg Sch\"{u}rmann.
\newblock {\em Topology of singular spaces and constructible sheaves},
  volume~63 of {\em Instytut Matematyczny Polskiej Akademii Nauk. Monografie
  Matematyczne (New Series) [Mathematics Institute of the Polish Academy of
  Sciences. Mathematical Monographs (New Series)]}.
\newblock Birkh\"{a}user Verlag, Basel, 2003.

\bibitem{Tho00}
Richard~P Thomas.
\newblock A holomorphic casson invariant for {C}alabi--{Y}au 3-folds, and
  bundles on {$ K3 $} fibrations.
\newblock {\em Journal of Differential Geometry}, 54(2):367--438, 2000.

\bibitem{Tod19}
Yukinobu Toda.
\newblock Categorical {D}onaldson--{T}homas theory for local surfaces.
\newblock {\em arXiv preprint arXiv:1907.09076}, 2019.

\bibitem{Ver81}
J.-L. Verdier.
\newblock Sp\'{e}cialisation de faisceaux et monodromie mod\'{e}r\'{e}e.
\newblock In {\em Analysis and topology on singular spaces, {II}, {III}
  ({L}uminy, 1981)}, volume 101 of {\em Ast\'{e}risque}, pages 332--364. Soc.
  Math. France, Paris, 1983.

\end{thebibliography}

\end{document}